\documentclass[12pt,reqno]{amsart}
\usepackage[top=1in, bottom=1in, left=1in, right=1in]{geometry}
\usepackage{amssymb,amsmath,amsthm,amsfonts,amsopn,url,color,enumerate,mathtools,microtype,MnSymbol,wrapfig}
\usepackage[colorlinks=true,linkcolor=blue,urlcolor=blue,citecolor=blue,pagebackref=true]%
  {hyperref}
\usepackage{bm}
\usepackage{mathbbol}
\usepackage{mathrsfs}
\usepackage[normalem]{ulem}
\usepackage[all]{xy}
\input xy
\xyoption{all}
\usepackage{amscd}
\usepackage{tikz,tikz-cd,comment}
\usetikzlibrary{decorations.pathmorphing}

\numberwithin{equation}{section}

\newtheorem{thm}[equation]{Theorem}
\newtheorem{lem}[equation]{Lemma}
\newtheorem{prop}[equation]{Proposition}

\theoremstyle{definition}
\newtheorem{defn}[equation]{Definition}
\newtheorem{notation}[equation]{Notation}
\newtheorem{examplex}[equation]{Example}
\newenvironment{example}
  {\pushQED{\qed}\examplex}
  {\popQED\endexamplex}
\newtheorem{rmkx}[equation]{Remark}
\newenvironment{rmk}
  {\pushQED{\qed}\rmkx}
  {\popQED\endrmkx}

\DeclareSymbolFontAlphabet{\mathbb}{AMSb}
\DeclareSymbolFontAlphabet{\mathbbl}{bbold}

\newcommand\<{\langle}
\renewcommand\>{\rangle}

\newcommand\II{\ensuremath{\mathbb I}}
\newcommand\NN{\ensuremath{\mathbb N}}

\newcommand\RR{\ensuremath{\mathbb R}}
\newcommand\ZZ{\ensuremath{\mathbb Z}}
\newcommand\cC{\ensuremath{\mathscr C}}
\newcommand\cD{\ensuremath{D}}

\newcommand{\bba}{\ensuremath{\bm{a}}}
\newcommand{\boldb}{\ensuremath{\bm{b}}}
\newcommand{\bbc}{\ensuremath{\bm{c}}}

\newcommand{\bbe}{\ensuremath{\bm{e}}}
\newcommand{\bbj}{\ensuremath{\bm{j}}}
\newcommand{\bbm}{\ensuremath{\bm{m}}}
\newcommand{\boldzero}{\ensuremath{\bm{0}}}
\newcommand{\facets}{\mathscr{F}} 
\newcommand{\fan}{\Sigma} 
\newcommand{\kk}{\mathbbl{k}} 
\newcommand{\id}{\mathrm{id}} 
\newcommand{\inc}{\iota} 
\newcommand{\mc}{\mathscr{M}} 
 
\newcommand{\proj}{\pi} 
\newcommand{\tfr}{\kk[\mc]} 
\newcommand{\rcr}{$\kk$-algebra realized by retracts}
\newcommand{\rcrs}{$\kk$-algebras realized by retracts}

\DeclareMathOperator{\Hom}{Hom}

\DeclareMathOperator{\rad}{rad}
\DeclareMathOperator{\supp}{supp}
\DeclareMathOperator{\chara}{char}

\title[Differential operators, retracts \& toric face rings]{Differential operators, retracts, and toric face rings}
\author[Berkesch]{Christine Berkesch}
\address{School of Mathematics \\ University of Minnesota\\ Minneapolis, MN}
\email{cberkesc@umn.edu}
\author[Chan]{C-Y. Jean Chan}
\address{Department of Mathematics\\ Central Michigan University \\ Mt. Pleasant, MI}
\email{chan1cj@cmich.edu}
\author[Klein]{Patricia Klein}
\address{Department of Mathematics \\ Texas A\&M University \\ College Station, TX}
\email{pjklein@tamu.edu}
\author[Matusevich]{Laura Felicia Matusevich}
\address{Department of Mathematics \\ Texas A\&M University \\ College Station, TX}
\email{matusevich@tamu.edu}
\author[Page]{Janet Page}
\address{Department of Mathematics\\ North Dakota State University \\ Fargo, ND}
\email{janet.page@ndsu.edu}
\author[Vassilev]{Janet Vassilev}
\address{Department of Mathematics and Statistics \\ University of New Mexico \\ Albuquerque, NM}
\email{jvassil@math.unm.edu }

\thanks{
CB was partially supported by NSF Grants DMS 1661962 and 2001101. \\
\indent \ LFM was partially supported by the Simons Foundation Collaboration Grant for Mathematicians. \\
}

\date{}
\begin{document}
\begin{abstract}
We give explicit descriptions of rings of differential operators of toric face rings in characteristic $0$.  For quotients of normal affine semigroup rings by radical monomial ideals, we also identify which of their differential operators are induced by differential operators on the ambient ring.  Lastly, we provide a criterion for the Gorenstein property of a normal affine semigroup ring in terms of its differential operators.  

Our main technique is to realize the $\kk$-algebras we study in terms of a suitable family of their algebra retracts in a way that is compatible with the characterization of differential operators.  This strategy allows us to describe differential operators of any $\kk$-algebra realized by retracts in terms of the differential operators on these retracts, without restriction on $\chara(\kk)$.
\end{abstract}
\maketitle
\tableofcontents

\parskip .4em

\setcounter{section}{0}
\section*{Introduction}
\label{sec:intro}
Differential operators play a notable role in branches of mathematics as seemingly disparate as partial differential equations and local cohomology, dynamical systems and invariant theory. 
Recently, there have been many exciting new developments concerning differential operators and their applications in commutative algebra, including to topics such as Bernstein--Sato polynomials, connections between singularities and local cohomology, equivariant $D$-modules, and Hodge ideals (see, for example,~\cite{DGJ20,JS20,HJJacCrit,QGBSatoMon,QGBernSato,AHN+21,JQGNBBSato,BJN19,Rai16, MP19,LWEquiDmod,LRLNumsDRings,PeEquihyper,Per20,PR21HodgeIdeal,LPEquiDmod}).  
One obstruction to the even greater use of rings of differential operators is the notorious difficulty of computing them explicitly. 
In fact, there are very few classes of rings whose differential operators are systematically computed, namely polynomial rings, Stanley--Reisner rings, affine semigroup rings, and coordinate rings of curves (see, for example, ~\cite{Tripp,TrDM,Er_SRDops,Mussontori,Jones, Mussontv,MuVDBTor,smoothToricDmods, Sai-Tr-DASR,Sai-Tr-fgdo,Muhasky,MussonMorita,SmStDO,Er}). 
The goal of this article is to give a new class of explicitly computed rings of differential operators.

A \emph{retract} of a $\kk$-algebra $R$ is a subring of $R$ which is isomorphic to a quotient of $R$. 
We call $R$ a \emph{\rcr} if it can be embedded into a finite direct sum of domains, each of which is a retract of $R$. 
Algebra retracts are of considerable interest; it is common in the literature to prove desirable properties of retracts using information from the ambient ring (see, for example,~\cite{Cos77,herzog77,CH97,BG02,Ter02,Gup14,EpNgAlgRet}). 
In this work, we do the opposite; we use knowledge about the differential operators on algebra retracts to compute differential operators of an ambient ring (see Theorems~\ref{thm:firstCharacterization} and~\ref{thm:secondCharacterization}). 
Naturally, this approach is most productive when the differential operators on the retracts are known. 
This is the case for a combinatorially interesting class of {\rcrs} known as \emph{toric face rings}.

Toric face rings were first introduced by Stanley \cite{Stanley87-tfr}, and further developed in~\cite{BG02,IcRotfr,IcRoCMtfr,OkYaDCtfr,Ngseminormlctfr,NgKPtfr,YaDCtfr}, among others. 
They include both Stanley--Reisner rings and affine semigroup rings as special cases, thus bringing under a single umbrella two of the mainstays of combinatorial commutative algebra.
In this setting, the retracts we are interested in are affine semigroup rings, whose rings of differential operators are known in characteristic zero; here we use the presentation given by Saito and Traves in~\cite{Sai-Tr-DASR,Sai-Tr-fgdo}. 
Thus we can directly apply Theorem~\ref{thm:secondCharacterization} to compute the ring of differential operators of a toric face ring in Proposition~\ref{prop:tfr-diffops}.
As a consequence, we recover results on differential operators on Stanley--Reisner rings  given by Tripp, Eriksson, and Traves in \cite{Tripp,Er_SRDops,TrDM} over arbitrary fields. 

Theorem~\ref{thm:firstCharacterization} finds the differential operators of a {\rcr} as a subring of the direct sum of the rings of differential operators of the retracts. 
This direct sum of the retracts is in general much larger than the original ring. On the other hand, the differential operators on Stanley--Reisner rings and affine semigroup rings are induced from differential operators over their natural ambient rings.
In general, however, the richness of the direct sum is really needed. 
To illustrate this, we provide a description of which differential operators on a quotient of a normal affine semigroup ring by a radical monomial ideal are induced from the differential operators on the ambient semigroup ring (see Theorem~\ref{thm:normal-diffops}) and show that this does not necessarily give the whole ring of differential operators. 

\subsection*{Outline}
In Section \ref{sec:DopsalgebraRetracts}, we describe the ring of differential operators of a \rcr\ in terms of differential operators on these retracts. 
In Section~\ref{sec:apps}, we apply the results of Section \ref{sec:DopsalgebraRetracts} in characteristic zero to compute the rings of differential operators of Stanley--Reisner rings and of toric face rings more generally.
In Section~\ref{sec:SmStmod}, we consider the quotient of a normal affine semigroup ring $R$ by a radical monomial ideal $J$ in characteristic zero and provide an explicit formula for the differential operators on $R/J$ that are induced by operators on $R$. 
Finally, in Section \ref{sec:Gorenstein}, 
we use differential operators to provide a new condition for the Gorenstein property of an affine normal semigroup ring. 

Throughout this paper, let $\kk$ denote an algebraically closed field; in Sections \ref{sec:apps}, \ref{sec:SmStmod}, and \ref{sec:Gorenstein}, assume that the characteristic of $\kk$ is zero. 

\subsection*{Acknowledgements}
We would like to thank the organizers, Karen Smith, Sandra Spiroff, Irena Swanson, and Emily Witt, of \emph{Workshop: Women in Commutative Algebra (WICA)}, where this work commenced; additionally, we appreciate the hospitality of BIRS in hosting this meeting. 
We thank AIM for supporting our collaboration, which greatly helped the advancement of this project. 
We are grateful to Jack Jeffries for valuable discussions that clarified our results; 
we also thank Monica Lewis, Pasha Pylyavskyy, and Vic Reiner for helpful conversations.  Finally, we thank the anonymous referee for feedback on an earlier version of our manuscript.  

\section{\texorpdfstring{\rcrs}{k-algebras realized by retracts}}
\label{sec:algebraRetracts}
In this section we introduce \emph{\rcrs} and provide examples. 
Let $S$ and $R$ be $\kk$-algebras. 
An injective $\kk$-algebra homomorphism $\iota\colon S \to R$ is called an \emph{algebra retract} of $R$ if there exists a surjective homomorphism of $\kk$-algebras $\pi\colon R \to S$ such that $\pi \circ \iota = \id_S$. 
When $S$ and $R$ are graded, we also assume that the homomorphisms $\iota$ and $\pi$ are also graded. 
See \cite{herzog77} for a local version of this definition.

\begin{defn}
\label{def:rcr}
Let $R$ be a $\kk$-algebra, and
let $\{S_\ell \}_{\ell \in \Lambda}$ be a finite collection of domains that are algebra retracts of $R$ with defining $\kk$-algebra homomorphisms $\iota_\ell\colon S_\ell \hookrightarrow R$ and $\pi_\ell\colon R\twoheadrightarrow S_\ell$ satisfying $\pi_\ell \circ \iota_\ell=\id_{S_\ell} $. For each $\ell$, let $P_\ell := \ker \pi_\ell$. 
We call $R$ a \emph{\rcr} and say that $R$ is \emph{realized by the retracts $\{S_\ell \mid \ell \in \Lambda\}$} when the following two conditions hold. 
\begin{enumerate}
\item[(i)] 
The following map is injective: 
\begin{equation} 
\label{eqn:intoretractsum}
    \phi\colon R \rightarrow \bigoplus\limits_{\ell \in \Lambda} S_\ell 
    \quad \text{given by} \quad 
    f \mapsto (\pi_\ell (f) \mid \ell \in \Lambda); 
\end{equation} 
in other words, $\displaystyle \bigcap_{\ell \in \Lambda} P_\ell = 0$.
\item[(ii)] 
The $S_\ell$ are irredundant for \eqref{eqn:intoretractsum}; 
more precisely, for each $i \in \Lambda$, we have $\displaystyle \bigcap_{\substack{\ell \neq i \\ \ell \in \Lambda}} P_\ell \neq 0$. 
\end{enumerate}
Since $\iota_\ell$ is injective, if $f \in S_\ell$, we write $f=\iota_\ell(f)$ in $R$.  For $f \in R$, we use $\phi(f)_i$ to denote the $i^{th}$ coordinate of $\phi(f)$ in $\bigoplus\limits_{\ell \in \Lambda} S_\ell$.
\end{defn}

For example, suppose that $R$ is a quotient of a $\kk$-algebra $T$ by a radical ideal $I$ with associated primes $P_1,\dots,P_r$. 
Now if $T/P_i$ are algebra retracts of $R$ for all $1\leq i\leq r$, then $R$ is realized by the retracts $\{S_\ell\mid \ell\in\Lambda\} = \{T/P_1,T/P_2,\dots,T/P_r\}$. 
In the remainder of this section, we show that Stanley--Reisner rings and toric face rings are \rcrs.

\subsection{Stanley--Reisner rings as \texorpdfstring{\rcrs}{k-algebras realized by retracts}}
\label{subs:SRringsasrcrs}
Let $\Delta$ be a simplicial complex on a finite vertex set $V = \{1,2,\dots,d\}$. 
The \emph{Stanley--Reisner ring of $\Delta$} is the $\kk$-algebra given by 
\[
\kk[\Delta] := \frac{\kk[t_1,t_2,\dots,t_d]}{\<t^{\bba}\mid \bba \in\NN^d, \supp(\bba)\notin\Delta\>},
\]
where $t^{\bba} = t_1^{a_1}t_2^{a_2}\cdots t_d^{a_d}$ and 
$\supp(\bba) := \{i\in V\mid a_i\neq 0\}$. Retracts have been previosly studied in this context. For example,
in \cite{EpNgAlgRet}, Epstein and Nguyen show that every graded algebra retract of a Stanley--Reisner ring $\kk[\Delta]$ is a Stanley--Reisner ring. 
Further, all such retracts are isomorphic to $\kk[\Delta|_W]$, where $\Delta|_W$ is the restriction of $\Delta$ to a subset $W$ of $V$. 

One way to view $\kk[\Delta]$ as a \rcr\ is via the facets of $\Delta$. 
If $\{F_\ell\}_\ell\in\Lambda$ denotes the collection of facets of $\Delta$, then 
$\kk[F_\ell] \cong \kk[t_i\mid i \in F_\ell]$, 
and  
\begin{align*}
\iota_\ell\colon \kk[F_\ell]\hookrightarrow \kk[\Delta] 
\quad\text{and}\quad 
\pi_\ell\colon \kk[\Delta]&\twoheadrightarrow \kk[F_\ell]\\
t_i &\mapsto 
    \begin{cases} 
    t_i &\text{if }i\in F_\ell,\\ 
    0 &\text{otherwise} 
    \end{cases}
\end{align*}
are the maps needed to see that 
\[
\phi\colon \kk[\Delta] \to \bigoplus_{\ell\in\Lambda} \kk[F_\ell] 
\quad\text{given by}\quad 
t_i\mapsto (\pi_\ell(t_i)\mid \ell\in\Lambda) 
\]
is an injective map. Since the retracts are domains, this implies that $\kk[\Delta]$ is a \rcr. 

\subsection{Toric face rings as \texorpdfstring{\rcrs}{k-algebras realized by retracts}}
\label{subs:tfrasrcrs}
The building blocks of toric face rings are affine semigroup rings, so we begin with those. \begin{notation}
\label{not:affine-semigroup-ring}
Let $M$ be a finitely generated submonoid of $\ZZ^d$.
The \emph{affine semigroup ring defined by $M$} is
\[
\kk[M] 
:= \bigoplus_{\bba\in M} \kk \cdot {t}^{\bba}, 
\]
where ${t}^{\bba} = t_1^{a_1}t_2^{a_2} \cdots t_d^{a_d}$ for ${\bba}= (a_1,a_2, \ldots, a_d)$.
\end{notation}

Finitely generated submonoids of $\ZZ^d$ are usually called affine semigroups. It is convenient to 
view the generators of an affine semigroup $M$ as the columns of a $d \times n$ integer matrix $A$; in this case, we use the notation $M = \mathbb{N}A$.
Throughout this article, we assume that the group generated by the columns of $A$ is the full ambient lattice, so $\ZZ A = \ZZ^d$ and also that the real positive cone over $A$, $\RR_{\geq 0} A$, is pointed, meaning that it contains no lines.  

A semigroup $\NN A$ is \emph{normal} if $\NN A = \RR_{\geq 0} A \cap\ZZ A$. In this case, the semigroup ring $\kk[\NN A]$ is normal in the sense of commutative algebra.

A hyperplane $H$ in $\RR^d$ is a supporting hyperplane of $\RR_{\geq 0}A$ if this cone lies entirely in one of the closed half spaces defined by $H$.
A \emph{face} $\sigma$ of $\RR_{\geq 0} A$ (or $A$ or $\NN A$) is the intersection of $\NN A$ with a supporting hyperplane of $\RR_{\geq 0}A$. 
Such a face is called a \emph{facet} if its linear span has dimension $d-1$. 
This is somewhat nonstandard terminology, as our faces and facets are submonoids of $\NN A$ instead of cones. 

Recall that the $\ZZ^d$-graded prime ideals in $\kk[\NN A]$ are in one-to-one correspondence with the faces of $A$ (or $\RR_{\geq 0} A$)~ \cite[Proposition 1.5]{Ishida}, as a face $\tau$ of $A$ corresponds to the multigraded prime $\kk[\NN A]$-ideal 
$P_\tau = \left\<t^{\bba}\, \big\vert \,  \bba\in\NN A\setminus \tau \right\>$.

Next, let $\fan \subset \RR^d$ be a rational polyhedral fan consisting of strongly convex (or pointed) cones. 
A \emph{monoidal complex} $\mc$ supported on $\fan$ is a collection of monoids $\{M_\sigma \mid \sigma \in \fan\}$ such that
\begin{itemize}
    \item[ (i)] $M_\tau \subseteq \tau \cap \ZZ^d$ and $\RR_{\geq 0} M_\tau = \tau$;
    \item[(ii)] if $\sigma, \tau \in \fan$ and $\sigma \subseteq \tau$, then $M_\sigma = \sigma\cap M_\tau$.
\end{itemize}
Denote $|\mc| = \bigcup_{\tau \in \fan} M_\tau$. 

\begin{defn}
\label{def:tfr}
The \emph{toric face ring} of $\mc$ over $\kk$ is given as a graded vector space by
\[ 
\tfr = \bigoplus_{\bba \in |\mc|} \kk \cdot t^{\bba}
\]
with multiplication defined by 
\[
t^{\bba} \cdot t^{\boldb} = 
\left\{
\begin{array}{ll} 
t^{\bba+\boldb}  &  \text{if there is } \tau \in \fan \text{ such that } \bba, \boldb \in \tau,\\
0 & \text{otherwise.}
\end{array}
\right.
\]
\end{defn}

Given $\tau \in \fan$, the semigroup ring $\kk[M_\tau]$ is a subring of $\tfr$; we denote by 
\[
\inc_{\tau}\colon \kk[M_\tau] \hookrightarrow \tfr
\]
the natural inclusion.
For $\tau \in \fan$, let 
\[
P_\tau := \< t^{\bba} \mid \bba \in |\mc| \setminus M_\tau \>.
\]
Then $\tfr/P_\tau$ is isomorphic to the semigroup ring $\kk[M_\tau]$, so that $P_\tau$ is a monomial prime ideal of $\tfr$. 
We identify $\tfr/P_\tau$ with $\kk[M_\tau]$, and denote by
\[
\proj_\tau\colon \tfr \twoheadrightarrow \kk[M_\tau]
\]
the natural projection onto the quotient. Clearly, $\kk[M_\tau]$ is a retract of $\tfr$.

A \emph{facet} of $\fan$ is a cone in $\fan$ which is maximal with respect to inclusion among all elements of $\fan$.
Let $\facets(\fan)$ denote the collection of all facets of $\fan$.
Since $\bigcap_{\tau \in \facets(\fan)} P_\tau = 0$, the ring homomorphism
\begin{equation}
    \label{eqn:intoSum}
\phi\colon \tfr \to \bigoplus_{\tau \in \facets(\fan)} \kk[M_\tau]  
\quad\text{given by} \quad  
f \mapsto(\proj_\tau(f) \mid \tau \in \facets(\fan))
\end{equation}
is injective. It follows that $\tfr$ is a {\rcr}.

In general, the algebra retracts of toric face rings of $\kk[\mc]$ are given by restricting $\mc$ to a subfan $\Gamma$ of $\Sigma$
~\cite[Proposition~4.5]{EpNgAlgRet}. 
The retracts that we consider here are those given by the restriction of $\Sigma$ to one of its maximal cones $\Gamma$, yielding an affine semigroup ring. 

\begin{example}
\label{rem:tfr-SR}
Stanley--Reisner rings of simplicial complexes are toric face rings. 
To see this, let $\Delta$ be a simplicial complex on the vertex set $V = \{1,2,\dots,d-1\}$. 
For each subset $F$ of $V$, associate the pointed cone $C_F$ generated by the set of elements of the form $\bbe_i + \bbe_d$ for $i \in F$, where $\bbe_i$ denotes the $i$-th standard basis vector in $\RR^d$. 
If $\Sigma$ denotes the fan in $\RR^d$ consisting of the cones $C_F$ for $F \in \Delta$, then the toric face ring $\kk[\mc]$ is isomorphic to the Stanley--Reisner ring $\kk[\Delta]$. 
\end{example}

\begin{example}
\label{ex:tfr-semigroup}
On the other hand, when $\Sigma$ has a unique maximal cone, then $\kk[\mc]$ is simply an affine semigroup ring. 
An affine semigroup ring $\kk[\NN A]$ modulo a radical monomial ideal $J$ is also a toric face ring. 
In this case, if the ideal $J = \bigcap_{i=1}^r P_{\tau_i}$, then the fan $\Sigma$ consists of faces of the cone $\RR_{\geq 0} A$ that are contained in $\tau_i$ for some $1\leq i\leq r$. 
We will examine this case more closely in Sections~\ref{sec:SmStmod} and~\ref{sec:Gorenstein}. 
\end{example}

\section{Rings of differential operators on \texorpdfstring{\rcrs}{k-algebras realized by retracts}}
\label{sec:DopsalgebraRetracts}
Fix a $\kk$-algebra $R$ and $R$-module $M$, and note that in this section we have no requirements on the characteristic of $\kk$ unless explicitly mentioned. 
The $\kk$-linear differential operators $D(R,M)$ are defined inductively by the degree $i$ of the operator. 
The degree $0$ differential operators are 
\[
\cD^0(R,M) := \Hom_R(R,M),
\] 
and, for $i>0$, the degree $i$ differential operators are 
\[
\cD^i(R,M) 
:= 
\{\delta \in \Hom_\kk(R,M) \mid [\delta,r] \in \cD^{i-1}(R,M) \text{ for all } r \in R\},
\]  
where $[\delta,r] := \delta \circ r - r \circ \delta$ is the commutator. 
The module of differential operators on $M$, denoted $\cD(R, M)$, is $\displaystyle \bigcup_{i \geq 0} D^i(R,M)$.  We write $\cD(R)$ for $\cD(R,R)$.  Note that $\cD(R,-)$ is a left-exact functor, and that, in particular, for any $R$-ideal $J$, $\cD(R, J) = \{\delta \in \cD(R) \mid \delta ( R )\subset J \}$.

\begin{example}
If $\chara(\kk)=0$, the ring of differential operators of the polynomial ring in $d$ variables with coefficients in $\kk$ is the Weyl algebra 
\[
W=\kk[t_1,\ldots,t_d]\< \partial_1, \ldots, \partial_d \>,
\]
where the relations defined on the generators are $t_it_j-t_jt_i=0=\partial_i\partial_j-\partial_j\partial_i$, and $\partial_it_j-t_j\partial_i=\delta_{ij}$, the Kronecker-$\delta$ function.  The ring of differential operators of the ring of Laurent polynomials in $d$ variables is the extended Weyl algebra $\kk[t_1^{\pm 1},\ldots,t_d^{\pm 1}]\< \partial_1, \ldots, \partial_d\>$ with the relations as in the ordinary Weyl algebra together with the additional relation $ t_i^{-1}\partial_j-\partial_jt_i^{-1}=t_i^{-2}\delta_{ij}$. 
\end{example}

Let $R$ be a $\kk$-algebra realized by the retracts $\{S_\ell \mid \ell \in \Lambda\}$, as in Definition~\ref{def:rcr}.
Given $\ell \in \Lambda$, and $\delta \in D(R)$, set $\delta_\ell=\pi_\ell \circ \delta\circ \iota_\ell$. In this section, we show that the  map:
\begin{equation}
\label{eqn:intoDiffRetracts}
\cD(R) \to \bigoplus_{\ell\in \Lambda} \cD(S_\ell)
\quad \text{given by}\quad
\delta \mapsto (\delta_\ell := \proj_\ell \circ \delta \circ \inc_\ell \mid \ell \in \Lambda),
\end{equation} 
is injective and
compute the ring of differential operators $\cD(R)$ in terms of the rings of differential operators 
$\cD(S_\ell)$ for $\ell \in \Lambda$.
We  do this in two ways (see Theorems~\ref{thm:firstCharacterization} and~\ref{thm:secondCharacterization}). First, we include two lemmas.

\begin{lem}
\label{lemma:partialOp1}
 Assume that the $\kk$-algebra $R$ is realized by the retracts $\{S_\ell \mid \ell \in \Lambda\}$. Let $\delta \in \cD^i(R)$. If $\ell \in \Lambda$, then $\delta_\ell \in \cD^i(S_\ell)$.
\end{lem}

\begin{proof}
We use induction on $i$. 
If $\delta \in \cD^0(R)=\Hom_{R}(R,R)$, then $\delta$ is given by multiplication by a fixed element of $R$, say $r$. 
But then $\delta_\ell$ is given by multiplication by $\proj_\ell(r)$, and so $\delta_\ell \in \cD^0(S_\ell)$.

Now assume the result is true for operators of order $i-1$, and let $\delta \in \cD^i(R)$. 
The $\kk$-linearity of $\delta_\ell$ follows from $\kk$-linearity of $\delta$, since $\inc_\ell$ and $\proj_\ell$ are $\kk$-linear.
To show that $\delta_\ell \in \cD^i(S_\ell)$,
it is enough to verify that $[\delta_\ell,f] \in \cD^{i-1}(S_\ell)$ for all $f \in S_\ell$.
This follows by induction since, for $f \in S_\ell$, $[\delta_\ell,f] = [\delta,\iota_\ell(f)]_\ell$.  To see this, note that $\iota_\ell \circ f$ is the map $\iota_\ell(f)\iota_\ell$ and for any $g \in R$, $\pi_\ell \circ g$ is the map $\pi_\ell(g) \pi_\ell$.  Then we have
\begin{align*}
[\delta_\ell,f] &= \pi_\ell \circ \delta \circ \iota_\ell \circ f - f \circ \pi_\ell \circ \delta \circ \iota_\ell \\
&= \pi_\ell \circ \delta \circ \iota_\ell(f) \iota_\ell - \pi_\ell(\iota_\ell(f)) \pi_\ell \circ \delta \circ \iota_\ell \\
&= \pi_\ell \circ \delta \circ \iota_\ell(f) \iota_\ell - \pi_\ell \circ \iota_\ell(f)\delta \circ \iota_\ell \\
&=\pi_\ell \circ \left(\delta \circ \iota_\ell(f)  -  \iota_\ell(f)\delta \right) \circ \iota_\ell \\
&= [\delta, \iota_\ell(f)]_\ell.
\end{align*}
\end{proof}

\begin{lem}
\label{lemma:partialOp2}
Assume that the $\kk$-algebra $R$ is realized by the retracts $\{S_\ell \mid \ell \in \Lambda\}$.
Let $\delta \in \cD^i(R)$, $\ell \in \Lambda$, and $f \in R$.  
If $\proj_\ell(f) = 0$, then $\proj_\ell(\delta(f)) = 0$.
\end{lem}

\begin{example}
\label{ex:2vars}
Lemma~\ref{lemma:partialOp2} is a key ingredient for the main results of this section and provides an easily checked necessary condition for a linear operator on $R$ to be a differential operator on $R$. To see this more concretely, consider $R=\kk[x,y]/\<xy\>$. This Stanley--Reisner ring is realized by the retracts $\kk[x]\cong R/\langle y\rangle$ and $\kk[y]\cong R/\langle x\rangle$, where the canonical projections are denoted $\proj_x$ and $\proj_y$, respectively. The ring of differential operators on $R$ is generated as a $\kk$-algebra by the operators $x^i\partial_x^j$ and $y^k\partial_y^\ell$ such that $i\geq j$ and $k \geq \ell$. To see why $\partial_x$ is not a differential operator on $R$, note that $\proj_y(x)=0$, but
$\proj_y(\partial_x(x)) = \proj_y(1) \neq 0$, contradicting the conclusion of Lemma~\ref{lemma:partialOp2}.
\end{example}

\begin{proof}[Proof of Lemma~\ref{lemma:partialOp2}]
By induction on $i$ as before, if $i=0$, then 
$\delta(f) = r \cdot f$ for a fixed $r\in R$. 
If $\proj_\ell(f) = 0$, it follows that
$\proj_\ell(r\cdot f) = 0$, since $\proj_\ell$ is a ring homomorphism.

Now assume the result is true for operators of order $i-1$, and let $\delta \in \cD^i(R)$.
Suppose $\proj_\ell(f) = 0$ and let $P = \displaystyle \bigcap_{f \notin P_j} P_j$ where $P_j=\ker \pi_j$. Let $g$ be a nonzero element in $P$ such that $\pi_\ell(g) \neq 0$.  Such an element exists since $R$ is realized by the retracts $\{S_\ell\mid \ell\in\Lambda\}$.
Then since $f \cdot g \in \left( \displaystyle \bigcap_{f \in P_i} P_i \right) \left( \displaystyle \bigcap_{f \notin P_j} P_j \right) \subseteq \displaystyle \bigcap_{\ell \in \Lambda} P_\ell =0$, we have $f \cdot g =0$ in $R$.  
By the inductive hypothesis, $\proj_\ell([\delta,g](f))=0$. 
Then
\[
[\delta,g](f) = \delta(g \cdot f) - g \delta(f) = - g \delta(f),
\]
where the last equality holds because $\delta(g \cdot f)=\delta(0)=0$. 
Hence
\[
0 = \proj_\ell([\delta,g](f)) = \proj_\ell(- g\delta(f)) = - \proj_\ell(g) \proj_\ell(\delta(f)).
\]
As $g$ was chosen so that $\proj_\ell(g) \neq 0$, it follows that $\proj_\ell(\delta(f))=0$ since $S_\ell$ is a domain.
\end{proof}

We are now ready to give a first description of $D(R)$.

\begin{thm}
\label{thm:firstCharacterization}
 Assume that the $\kk$-algebra $R$ is realized by the retracts $\{S_\ell \mid \ell \in \Lambda\}$ with the injective map 
 $\phi\colon R \rightarrow \bigoplus\limits_{\ell \in \Lambda} S_\ell$.
 Let $\delta\colon R \to R$ be $\kk$-linear. Then $\delta \in \cD^i(R)$ if and only if the following two conditions hold:
\begin{enumerate}
    \item $\delta_\ell \in \cD^i(S_\ell)$ for all $\ell \in \Lambda$; 
    \label{cond:1}
    \item If $\ell \in \Lambda$ and 
    $\proj_\ell(f) = 0$, 
    then $\proj_\ell(\delta(f)) = 0$.
    \label{cond:2}
\end{enumerate}
\end{thm}

\begin{proof}
If $\delta \in \cD^i(R)$, then $\delta$ satisfies Conditions~\eqref{cond:1} and~\eqref{cond:2} by Lemmas~\ref{lemma:partialOp1} and~\ref{lemma:partialOp2}.

Now assume that $\delta$ is $\kk$-linear and satisfies Conditions~\eqref{cond:1} and~\eqref{cond:2}.
Fix $f\in R$. For each $\ell \in \Lambda$, set $f_\ell = \iota_\ell(\pi_\ell(f))$. Then $\pi_\ell(f) = \pi_\ell(f_\ell)$ for all $\ell \in \Lambda$, and so $\pi_\ell(f-f_\ell)=0$ for all $\ell \in \Lambda$.  By 
Condition~\eqref{cond:2}, we have 
$\pi_\ell(\delta(f-f_\ell))=0$.  Hence, 
\begin{equation}\label{eq:cond2observation}
\pi_\ell(\delta(f))=\pi_\ell(\delta(f_\ell)) \quad \text{for each } \ell \in \Lambda.
\end{equation}

We now show that $\delta \in \cD^i(R)$ by induction on $i$.
Set $r=\delta(1)$, and assume that, for $\ell \in \Lambda$, $\delta_\ell \in \cD^0(S_\ell)$ by Condition~\eqref{cond:1}.  
Note that $\delta_\ell = \pi_\ell \circ \delta \circ \iota_\ell$.
As an $S_\ell$-module homomorphism on $D^0(S_\ell)$, $\delta_\ell$ is determined by the image of the identity element $1_{S_\ell}$ 
in $S_\ell$ under $\delta_\ell$. 
That is, 
for any $g \in S_\ell$, 
\[
\delta_\ell(g) = \delta_\ell(1_{S_\ell} \cdot g) = \delta_\ell(1_{S_\ell}) g.
\]
Applying (\ref{eq:cond2observation}) to $f=1$ in $R$, we have
$\pi_\ell(\delta(1_{\ell})) 
= \pi_\ell(\delta(1)) 
= \pi_\ell(r)$. 
On the other hand,
\[ 
\pi_\ell(\delta(1_{\ell})) 
= (\pi_\ell \circ \delta)( \iota_\ell ( \pi_\ell ( 1)))
= (\pi_\ell \circ \delta \circ \iota_\ell) (\pi_\ell(1)) 
= \delta_\ell( \pi_\ell(1) ) 
= \delta_\ell( 1_{S_\ell} ). 
\]
It follows that $\delta_\ell$ is given by multiplication by $\proj_\ell(r)$ for each $\ell \in \Lambda$. 
Next note that
\[ 
\pi_\ell(\delta(f_\ell)) 
= \pi_\ell \circ \delta \circ \iota_\ell(\pi_\ell(f)) 
= \delta_\ell (\pi_\ell (f)) 
= \pi_\ell(r) \pi_\ell(f) 
= \pi_\ell(rf).
\]

Combining these two observations, we have,
by definition of $\phi$,
\[
\phi(\delta(f)) = (\pi_\ell (\delta (f)))_{\ell \in \Lambda}
= (\pi_\ell(rf))_{\ell \in \Lambda}=\phi(rf).
\]  
Since $\phi$ is injective, $\delta(f)=rf$, which implies that $\delta \in \cD^0(R)$.

Assume now that the result is true for operators of order $i-1$, and let $\delta$ satisfy the required conditions for order $i$: 
\eqref{cond:1} $\delta_\ell \in D^i(S_\ell)$ for all $\ell \in \Lambda$ and \eqref{cond:2} if $\ell \in \Lambda$ and $\pi_\ell(g)=0$ then $\pi_\ell(\delta(g))=0$.  

 We need to show that $[\delta,f] \in D^{i-1}(R)$ for $f \in R$.  By induction, it suffices to show that Conditions \eqref{cond:1} and \eqref{cond:2} hold for $[\delta, f]$.  To address Condition \eqref{cond:1}, fix $\ell \in \Lambda$. Since $[\delta,f]_\ell=\pi_\ell \circ [\delta,f] \circ \iota_\ell$, we will show that $\pi_\ell \circ [\delta,f] \circ \iota_\ell \in D^{i-1}(S_\ell)$ for all $\ell \in \Lambda$.
Let us now expand:
\begin{align*}
\pi_\ell \circ [\delta,f] \circ \iota_\ell &= \pi_\ell\circ (\delta \circ f-f \delta)\circ \iota_\ell\\
&= \pi_\ell\circ \delta \circ f \circ \iota_\ell-\pi_\ell(f \delta)\circ \iota_\ell\\
&= \pi_\ell\circ \delta \circ f_\ell \circ \iota_\ell-\pi_\ell(f) \pi_\ell \circ\delta\circ \iota_\ell  \quad\text{ by Equation~\eqref{eq:cond2observation} }\\
&=\pi_\ell\circ \delta \circ \iota_\ell(\pi_\ell(f)) \circ \iota_\ell-\pi_\ell(f) \pi_\ell \circ\delta\circ \iota_\ell \quad \text{by definition of } f_\ell \\
&=\pi_\ell\circ \delta \circ \iota_\ell \circ \pi_\ell(f) -\pi_\ell(f) \pi_\ell \circ\delta\circ \iota_\ell \quad \text{as  } \iota_\ell \text{ is a homomorphism} \\
&=\delta_\ell \circ \pi_\ell(f)-\pi_\ell(f) \delta_\ell \quad \text{by definition of } \delta_\ell \\
&=[\delta_\ell,\pi_\ell(f)].
\end{align*}
Since $\delta_ \ell \in D^i(S_\ell)$ by assumption, $[\delta_\ell,\pi_\ell(f)] \in D^{i-1}(S_\ell)$ by the definition of an order $i$ differential operator.

To address Condition \eqref{cond:2}, assume $\ell \in \Lambda$ and $\pi_\ell(g) = 0$ for some $g \in R$.  Because $\pi_\ell$ is a homomorphism and $\delta$ satisfies Condition \eqref{cond:2}, $\pi_\ell(fg) = \pi_\ell(f)\pi_\ell(g) = 0$, and so $\pi_\ell(\delta(fg)) = 0$.  Then \begin{align*} 
\pi_\ell \circ [\delta, f](g) = \pi_\ell(\delta(fg))-\pi_\ell(f  \delta(g)) = 0 - \pi_\ell(f)\pi_\ell(\delta(g))  = 0.
\end{align*}
Hence $[\delta,f]$ satisfied Condition \eqref{cond:2} as well as Condition \eqref{cond:1}, and so, by induction, $[\delta,f] \in D^{i-1}(R)$, and then $\delta \in D^i(R)$.
\end{proof}

To state our second characterization of $D(R)$, if $\lambda \subset \Lambda$, we need some additional notation. Set $S_\lambda= \bigcap\limits_{\ell \in \lambda} S_\ell$.  Since $S_\ell$ is an algebra retract of $R$ for all $\ell \in \lambda$, $S_\lambda$ is also an algebra retract of $R$ in a natural way (and indeed $S_\lambda$ is also an algebra retract of $S_\ell$ for all $\ell \in \lambda$).  We define $\inc_{\lambda,\ell}$ to be the natural inclusion that identifies $S_\lambda$ as a subring of $S_\ell$, and $\proj_{\ell,\lambda}$ is the natural projection from $S_\ell$ to $S_\lambda$, where the latter is considered as a quotient of the former.   

\begin{thm}  
\label{thm:secondCharacterization}
Assume that the $\kk$-algebra $R$ is realized by the retracts $\{S_\ell \mid \ell \in \Lambda\}$.
The map
\begin{equation*}
\cD(R) \to \bigoplus_{\ell\in \Lambda} \cD(S_\ell)
\quad \text{given by}\quad
\delta \mapsto (\delta_\ell = \proj_\ell \circ \delta \circ \inc_\ell \mid \ell \in \Lambda)
\end{equation*}
from~ \eqref{eqn:intoDiffRetracts}
is an injective ring homomorphism. 
A tuple $(\rho_\ell \mid \ell \in \Lambda) \in \bigoplus_{\ell\in \Lambda} \cD(S_\ell)$ is in the image of~\eqref{eqn:intoDiffRetracts} 
if and only if it satisfies the following two conditions:
\begin{enumerate}[(a)]
    \item if $\lambda \subset  \Lambda$ and $j,k \in \lambda$,  then $ \proj_{j,\lambda} \circ \rho_j \circ \inc_{\lambda,j} = \proj_{k,\lambda} \circ \rho_k \circ \inc_{\lambda,k}$; 
    \label{cond:a}
    \item 
    If $\lambda=\{j,k\}$, then $\proj_{j, \lambda} (\rho_j (f)) = 0$ if 
    $\proj_k(f) = 0$.
    \label{cond:b}
\end{enumerate}
\end{thm}

\begin{proof}
Since \eqref{eqn:intoDiffRetracts} is given by composition and direct sums of ring homomorphisms, it is itself a ring homomorphism.  

First, for $\delta \in \cD(R)$, its image satisfies Conditions~\eqref{cond:a} (by construction) and~\eqref{cond:b} by Condition~\eqref{cond:2} of Theorem~\ref{thm:firstCharacterization}.  
To prove injectivity, as well as verify the description of the image of~\eqref{eqn:intoDiffRetracts}, we construct an inverse.

For $\lambda \subset \Lambda$ and $\rho \in \cD(S_\lambda)$, set $\overline{\rho} = \inc_\lambda \circ \rho \circ \proj_\lambda$. 
When $\rho$ satisfies Condition~\eqref{cond:b}, then, by an argument similar to that at the end of the proof of Theorem~\ref{thm:firstCharacterization}, $\overline{\rho} \in \cD(R)$. 

If $(\rho_\ell \mid \tau \in \Lambda) \in \bigoplus_{\ell\in \Lambda} \cD(S_\ell)$ satisfies Condition~\eqref{cond:a} 
and $\varnothing \neq \lambda \subset \Lambda$, choose $\ell \in \lambda$, and set
$\rho_{\lambda} =  \proj_{\ell,\lambda} \circ \rho_\ell \circ \inc_{\lambda, \ell}$. By Condition~\eqref{cond:a}, $\rho_\lambda$ is independent of the choice of $\lambda \in \Lambda$.
Now let
\[
\rho =  \sum_{\varnothing \neq \lambda \subset \Lambda} (-1)^{|\lambda|-1} \overline{\rho_\lambda}.
\]
We claim that if $(\rho_\ell \mid \ell \in \Lambda) \in \bigoplus_{\ell\in \Lambda} \cD(S_\ell)$ satisfies Conditions~\eqref{cond:a} and~\eqref{cond:b}, then $\rho \in \cD(R)$, and the image of $\rho$ under the map~\eqref{eqn:intoDiffRetracts} is $(\rho_\ell \mid \ell \in \Lambda)$.
This gives the desired inverse and finishes the proof.
\end{proof}

\section{Applications}
\label{sec:apps}
In this section, we apply the results of Section~\ref{sec:DopsalgebraRetracts} to compute rings of differential operators for examples of \rcrs\ discussed in Section~\ref{sec:algebraRetracts}. 
Throughout this section, we assume that the characteristic of $\kk$ is zero. 

\subsection{Differential operators of Stanley--Reisner rings}
\label{subs:DopsSRrings}
Traves \cite{TrDM} gave a nice classification of the ring of differential operators for Stanley--Reisner ring. 
In particular, he proved that $D(\kk[\Delta])$ is generated as a $\kk$-algebra by
\[
\{ t^{\bba} \partial^{\boldb}  \mid t^{\bba} \in P \text{ or } t^{\boldb} \notin P \text{ for each minimal prime } P \text{ of } R\}.
\]
where $t^{\bba} \partial^{\boldb} = t_1^{a_1} \dots t_d^{a_d} \partial_1^{b_1} \dots \partial_d^{b_d}$ and $\partial_i = \frac{\partial}{\partial t_i}$.

One can see directly that this description matches our description of the ring of differential operators in terms of algebra retracts.  In particular, for a Stanley--Reisner ring $\kk[\Delta]$, recall that $\kk[\Delta]$ is realized by the retracts $\{\kk[F_\ell] \mid \ell \in \Lambda\}$, where $\{F_\ell \mid \ell \in \Lambda\}$ are the facets of $\Delta$.  Further, we know that the minimal primes of $\kk[\Delta]$ are exactly those corresponding to the facets, namely $P_{F_\ell} = \<t^{\bba} \mid \bba \in\NN^d, \supp(\bba)\notin F_\ell\>$ (see for example \cite{SRsurvey}). 

For any $\kk$-linear map $\delta: \kk[\Delta] \to \kk[\Delta]$, Theorem \ref{thm:firstCharacterization} tells us that $\delta \in D^i(\kk[\Delta])$ if and only if 
\begin{enumerate}
    \item $\delta_\ell \in D^i(\kk[F_\ell])$ for all $\ell \in \Lambda$ \label{cond:1-SR} and
    \item $\pi_\ell(\delta(f)) = 0$ for all $\ell \in \Lambda$ and $f \notin \kk[F_\ell]$.\label{cond:2-SR}
\end{enumerate}   
Now for any facet $F_\ell$ of $\Delta$, notice that $\kk[F_\ell] \cong \kk[t_i \mid i \in F_\ell]$, so that $D(\kk[F_\ell])$ is the standard Weyl algebra on the variables $\{t_i \mid i \in F_\ell\}$.  By Condition~\eqref{cond:1-SR}, $D(\kk[\Delta])$ must be generated by elements of the form $x^{\bba}\partial^{\boldb}$.

Further, we have that $\pi_\ell(t^{\bba}\partial^{\boldb}(f)) = 0$ for all $f \notin \kk[F_\ell]$ if and only if $t^{\bba} \notin \kk[F_\ell]$ or $\partial^{\boldb}(f) = 0$ for all $f \notin \kk[F_\ell]$.  This happens if and only if $\supp(\bba) \notin F_\ell$ or $\supp(\boldb) \in F_\ell$.  In other words, we have Condition~\eqref{cond:2-SR} if and only if $t^{\bba} \in P_{F_\ell}$ or $t^{\boldb} \notin P_{F_\ell}$ for every $\ell \in \Lambda$.

\subsection{Differential operators of toric face rings}
\label{subs:Dopstfr}
The algebra retracts of toric face rings that we will consider are affine semigroup rings. 
In~\cite{Sai-Tr-DASR}, Saito and Traves described the ring of differential operators for an affine semigroup ring $\kk[\NN A]$ over the complex numbers, when viewed as a subring of the ring of differential operators of the Laurent polynomials, 
i.e., 
\[
D(\kk[\ZZ^d])=\kk[t_1^{\pm 1}, \ldots, t_d^{\pm 1}]\< \partial_1, \ldots, \partial_d\>,
\]
where $\partial_i$ denotes the differential operator $\frac{\partial}{\partial t_i}$.  
Setting $\theta_j=t_j\partial_j$ for $1 \leq j \leq d$ and noting that $\theta_i\theta_j=\theta_j\theta_i$ for all $i$ and $j$, the ring $\kk[\theta]=\kk[\theta_1, \theta_2 \ldots, \theta_d]$ is a polynomial ring. Set 
\[
\Omega({\bbm}) 
:= \{{\bba} \in \NN A 
\mid {\bba} +{\bbm} \notin \NN A\} 
= \NN A \setminus (-{\bbm}+ \NN A). 
\]

The \emph{idealizer of $\Omega ({\bbm})$} is defined to be the $\kk[\theta]$-ideal
\[ 
\II( \Omega({\bbm}) ) 
:= \<f (\theta) \in \kk[\theta] \mid f(\bba)=0 \text{ for all } \bba\in \Omega({\bbm})\>, 
\]
with the $\theta_i$ of degree $\boldzero$. 
In fact, $\II( \Omega({\bbm}))$ consists of $f(\theta)$ such that $t^{\bbm} f(\theta) \in D(\kk[\NN A])$. This is a consequence of
Saito and Traves \cite[Theorem 2.1] {Sai-Tr-DASR}, where they show that 
\[
\cD(\kk[\NN A]) 
= \bigoplus\limits_{ {\bbm} \in \ZZ^d} t^{\bbm} \cdot \II(\Omega({\bbm})).
\]

To compute $\II(\Omega({\bbm}))$ for a normal semigroup ring, consider a facet $\sigma$ of $A$, recalling that by this we mean a submonoid of $\NN A$ whose linear span has dimension $d-1$.
The {\em primitive integral support function} (or simply {\em support function}) $F_{\sigma}$ is the unique linear form on $\RR^d$ such that 
\[
(1)\ F_{\sigma}(\RR_{\geq 0}A) \geq 0, \qquad (2)\ F_{\sigma}(\sigma)=0, \qquad \text{ and} \qquad (3)\ F_{\sigma}(\ZZ^d)=\ZZ.
\]

Saito and Traves used the $F_{\sigma}({\bbm})$ to determine the precise form of $\II(\Omega({\bbm}))$ for normal $\kk[\NN A]$.
The ideas of the computation in \cite{Sai-Tr-DASR} can be traced back to \cite{Jones} and \cite{Mussontori}.
Set 
\begin{equation}
\label{eqn:G}
G_{{\bbm}}(\theta):= 
\prod\limits_{F_{\sigma}({\bbm})<0}\prod\limits_{i=0}^{-F_{\sigma}({\bbm})-1}(F_{\sigma}(\theta)-i).
\end{equation}

\begin{thm}\cite[Theorem 3.2.2]{Sai-Tr-DASR} 
\label{thm:DRd}
Let $R$ be a normal affine semigroup ring of dimension $d$. 
Let $F_{\sigma_1}, \dots, F_{\sigma_{r}}$ be the support functions of the facets ${\sigma_i}$ of the semigroup defining $R$. 
Let ${\bbm} $ be a multidegree in $\ZZ^d$. Then
\[ 
\cD(R) _{\bbm} = 
t^{\bbm} \cdot \left\< G_{\bbm}(\theta) \right\>.
\]
\end{thm}

We now provide a graded description of 
$\cD^i(R)$ when $R=\tfr$ is a toric face ring.
Let $\sigma \in \Sigma$, and let $\rho \in \cD^i(\kk[M_\sigma])$. 
Then $\rho$ is a sum of operators of order $i$ each of which is homogeneous with respect to the natural $\ZZ\sigma$-grading (where $\ZZ\sigma$ is the abelian group generated by $M_\sigma$). 
We claim that if $\rho$ satisfies Condition~\eqref{cond:b} of Theorem~\ref{thm:secondCharacterization}, then each homogeneous component does as well.
This holds since applying an operator of multidegree $\boldb\in \ZZ\sigma$ to a monomial $t^{\bba}$ yields a (possibly zero) scalar multiple  of $t^{\bba+\boldb}$. As $\rho$ is a finite sum of operators of multidegrees $\boldb_1, \ldots \boldb_n$, $\rho(t^{\bba})=0$ if and only if the constant multiple of $t^{\bba+\boldb_i}$ is 0 for all $1 \leq i \leq n$. 

The following result then provides the final key to describe differential operators on toric face rings, as we illustrate in examples later.

\begin{prop}
\label{prop:tfr-diffops}
Let $\sigma \in \facets(\fan)$, and let $\rho \in \cD^i(\kk[M_\sigma])$ be homogeneous of degree $\boldb \in \ZZ\sigma$ satisfying Condition~\eqref{cond:b} from Theorem~\ref{thm:secondCharacterization}. Then $\rho = t^{\boldb} q(\theta)$ for some $q \in \kk[\theta]$ such that 
\begin{enumerate}[(i)]
\item $q(\bba) = 0$ if $\bba+\boldb \notin M_\sigma$ when $\bba \in M_{\sigma}$, and \label{cond:ainMsig}
\item For $\tau \in \facets(\fan) \setminus \{\sigma\}$ and $\bba\in M_\sigma\setminus M_{\sigma \cap \tau}$
such that $\bba+\boldb \in M_{\sigma \cap \tau}$, we have
$q(\bba)= 0$. \label{cond:q(a)=0}
\end{enumerate} 
\end{prop}
\begin{proof}
Condition \eqref{cond:ainMsig} follows from the description of $\cD(\kk[M_\tau])$. Condition \eqref{cond:q(a)=0} follows from the assumption that Condition~\eqref{cond:b} of Theorem ~\ref{thm:secondCharacterization} is satisfied for $\rho$.
\end{proof}

\begin{rmk}
\label{rmk:facetEquations}
The two conditions in the previous result arise from the multigraded nature of our descriptions of rings of differential operators.
For a given multidegree $\boldb$, operators of this degree are spanned by operators of the form $\delta = t^{\boldb}q(\theta)$, where $q$ is a polynomial. 
Then $\delta(t^{\bba}) = q(\bba) t^{\bba+\boldb}$. If $\bba$ belongs to our semigroup but $\bba+\boldb$ does not, then we must have $q(\bba)=0$ if $\delta$ is to be a differential operator on the semigroup ring. This is the first condition in Proposition~\ref{prop:tfr-diffops}.

Similarly, if we are working with a semigroup coming from a monoidal complex, it may happen that $\bba$ belongs to a semigroup, and $\bba+\boldb$ belongs to both the semigroup and some other semigroup in the complex.  In this case the projection of $\delta(t^{\bba})=q(\bba)t^{\bba+\boldb}$ onto this second semigroup will not vanish unless $q(\bba)=0$.
This explains the second condition in Proposition~\ref{prop:tfr-diffops}. 

For each $\boldb$, the set of $\bba$ at which the polynomials $q$ are forced to vanish lie on a finite collection of translates of the linear spans of the faces of the maximal cones in the fan $\fan$. This explains why their vanishing ideal is generated by products of shifts of linear forms corresponding to supporting hyperplanes of those faces. 
\end{rmk}

We now include concrete examples of differential operators in the toric face ring setting.

\begin{example}
\label{ex:gluedrationalnormalcurves}
Let 
$
A=\left[\begin{smallmatrix}
1 & 1 & 1 & 1 & 1\\
0 & 1 & 2 & 0 & 0\\
0 & 0 & 0 & 1 & 2\\
\end{smallmatrix}\right]$,
$B= \left[\begin{smallmatrix}
1 & 1 & 1 \\
0 & 1 & 2 \\
0 & 0 & 0 \\
\end{smallmatrix}\right]$
and
$C=\left[\begin{smallmatrix}
1  & 1 & 1\\
0  & 0 & 0\\
0  & 1 & 2\\
\end{smallmatrix}\right].$
Two of the  facets of the integral cone $\NN A$ are $\sigma= \NN B$ and $\tau=\NN C$.
Define 
\[
R:=\frac{\kk[\NN A]}{P_\sigma \cap P_\tau} 
= \frac{\kk[x,xy,xy^2,xz,xz^2]}{\<x^2yz,x^2y^2z,x^2yz^2,x^2y^2z^2\>}.
\]
Note that there is a natural map 
\[
\phi\colon R \rightarrow 
\frac{R}{P_\sigma} \oplus \frac{R}{P_\tau} 
\cong \kk[\NN B] \oplus \kk [ \NN C] 
\cong \kk[x,xy,xy^2] \oplus \kk[x,xz,xz^2].
\]
The differential operators defined on the face $\sigma \cap \tau$ in each multidegree $(u,0,0)$ can be realized by operators in $\kk[\NN B] \oplus \kk [ \NN C]$ in multidegree $((u,0),(u,0))$ generated by 
\begin{equation}
\label{eq:oprnc2}
(\rho_u,\delta_u) := 
\left(x^{u}\cdot \prod\limits_{i=0}^{-2u-1} (2\theta_x-\theta_y-i),
x^{u}\cdot \prod\limits_{i=0}^{-2u-1} (2\theta_x-\theta_z-i)\right)
\end{equation}
This is the case by Theorem~\ref{thm:secondCharacterization} since 
\[
\pi_{\sigma, \sigma\cap \tau} \circ \rho_u \circ \iota_{\sigma \cap \tau, \sigma} 
= x^{u} \cdot \prod\limits_{i=0}^{-2u-1} (2\theta_x-i)
= \pi_{\tau, \sigma\cap \tau} \circ \delta_u \circ \iota_{\sigma \cap \tau, \tau}
\]
and 
\[
\pi_{\sigma, \sigma\cap \tau}(\rho_u(x^uy^v))=0 \quad \text{ for } \quad (u,v) \in \sigma \setminus \tau,
\]
as well as
\[
\pi_{\tau, \sigma\cap \tau}(\rho_u(x^uz^v))=0 \quad \text{ for } \quad (u,v) \in \tau \setminus \sigma.
\]

The operators on $\sigma \setminus \sigma \cap \tau$ in each multidegree $(u,v,0)$ (for $v \neq 0$) can be realized by operators in $\kk[\NN B] \oplus \kk [ \NN C]$ in multidegree $((u,v),(0,0))$ generated by 
\begin{equation}
\label{eq:oprnc3}
    (\rho_{u,v}, 0) 
    := 
    \left(x^{u}y^v\cdot 
    \left(\prod\limits_{i=0}^{-2u+v-1} (2\theta_x-\theta_y-i)\right)\left(\prod\limits_{i=0}^{-v} (\theta_y-i)\right), 0\right).
\end{equation}

Since
\[
\pi_{\sigma, \sigma\cap \tau}(\rho_{u,v}(x^uy^v))=0 \quad \text{ for } \quad (u,v) \in \sigma,
\]
both Conditions \eqref{cond:a} and \eqref{cond:b} of Theorem~\ref{thm:secondCharacterization} are clearly satisfied.

A similar argument shows that the operators on $\tau \setminus \sigma \cap \tau$ in each multidegree $(u,0,v)$ (for $v \neq 0$) can be realized by operators in $\kk[\NN B] \oplus \kk [ \NN C]$ in multidegree $((0,0),(u,v))$ generated by 
\begin{equation}\label{eq:oprnc4}
(0,\delta_{u,v}):=\left(0,x^{u}z^{v}\cdot \left(\prod\limits_{i=0}^{-2u+v-1} (2\theta_x-\theta_z-i)\right)\left(\prod\limits_{i=0}^{-v} (\theta_z-i)\right)\right).
\end{equation}

Combinations of operators of the forms represented by \eqref{eq:oprnc2}, \eqref{eq:oprnc3} and \eqref{eq:oprnc4} also produce operators in our ring of differential operators for all $v \neq 0$.
\end{example}

\begin{example}\label{ex:pyramid}
Set 
$A=\left[\begin{smallmatrix} 
1 & 1 & 1 & 1\\
0 & 1 & 0 & 1\\
0 & 0 & 1 & 1\\
\end{smallmatrix}\right]$.  
Recall $\NN A$ gives the integral cone of the ring
$S=\kk[\NN A]=\kk [ x, xy,xz, xyz]$.  
The facets of $\NN A$ are 
\begin{align*}
\sigma_1 &= \NN \<\bbe_1,\bbe_1+\bbe_2\>, \ 
\sigma_3 = \NN \<\bbe_1+\bbe_3,\bbe_1+\bbe_2+\bbe_3\>, \\
\sigma_2 &= \NN \<\bbe_1,\bbe_1+\bbe_3\>, 
\sigma_4= \NN \<\bbe_1+\bbe_3,\bbe_1+\bbe_2+\bbe_3\>.
\end{align*}
The support functions of the facets of $\NN A$ are $\theta_z$, $\theta_y$, $\theta_x-\theta_y$, $\theta_x-\theta_y$, respectively.  
By Theorem~\ref{thm:DRd}, the differential operators on $R$ in multidegree $\bbm$ are given by 
\[
x^{m_1}y^{m_2}z^{m_3}\cdot 
\left\< 
\left(\prod\limits_{i=1}^{-m_3-1}(\theta_z-i)\right)\left(\prod\limits_{i=1}^{-m_2-1}(\theta_y-i)\right)\left(\prod\limits_{i=1}^{-m_1+m_3-1}(\theta_x-\theta_z-i)\right)\left(\prod\limits_{i=1}^{-m_1+m_2-1}(\theta_x-\theta_y-i)\right)
\right\>.
\]
Set 
\[
R= \frac{S}{P_{\sigma_1} \cap P_{\sigma_2} \cap P_{\sigma_3} \cap P_{\sigma_4}} 
= \frac{\kk[x,xy,xz,xyz]}{\<x^2yz\>}. 
\]
Note that $T:= \kk[a,b,c,d]/\<ad,bc\> \cong R$. 
We know the differential operators of $T$ by \cite{TrDM} since $T$ is a Stanley--Reisner ring. 
The generators of $D(T)$ in multidegree $-\bbe_1$ are 
\[
a^{-1}\theta_a(\theta_a-1), \ \ 
a^{-1}\theta_a\theta_b, \ \ 
a^{-1}\theta_a\theta_c.
\]
There are not enough operators in the extended Weyl algebra in three variable to express these three operators. 
Given the mappings
\[
\xymatrix{
D(T)  \ar[d]^{\cong} \ar[r] & D(\kk[a,b]) \oplus D(\kk[a,c]) \oplus D(\kk[b,d]) \oplus D(\kk[c,d]) \ar[d]^{\cong}\\
D(R) \ar[r]  & D(\kk[x,xy]) \oplus D(\kk[x,xz]) \oplus D(\kk[xy,xyz]) \oplus D(\kk[xz,xyz])}
\]
with horizontal maps as in Theorem~\ref{thm:secondCharacterization},
we see that 
\begin{align*}
    a^{-1}\theta_a(\theta_a-1)  &\mapsto (a^{-1}\theta_a(\theta_a-1), a^{-1}\theta_a(\theta_a-1),0,0)\\ &\quad\cong (x^{-1}\prod\limits_{i=0}^1(\theta_x-\theta_y-i),x^{-1}\prod\limits_{i=0}^1(\theta_x-\theta_z-i),0,0) \\
    a^{-1}\theta_a\theta_b  &\mapsto   (a^{-1}\theta_a\theta_b,0,0,0) \cong (x^{-1}(\theta_x-\theta_y)\theta_y,0,0,0)\\
    a^{-1}\theta_a\theta_c  &\mapsto  (0, a^{-1}\theta_a\theta_c,0,0) \cong (0,x^{-1}(\theta_x-\theta_z)\theta_z,0,0). 
\end{align*}
To get an operator acting as $ a^{-1}\theta_a(\theta_a-1), a^{-1}\theta_a\theta_b,$ or  $a^{-1}\theta_a\theta_c$ in terms of the linear support functions on $S$, we would need fractional expressions of the forms
\[
x^{-1}\cdot \dfrac{\prod\limits_{i=0}^1(\theta_x-\theta_y-i)\prod\limits_{i=0}^1(\theta_x-\theta_z-i)}{\prod\limits_{i=0}^1(\theta_x-i)}, x^{-1}\cdot \dfrac{(\theta_x-\theta_y)\theta_y(\theta_x-\theta_z)}{\theta_x},
x^{-1}\cdot \dfrac{(\theta_x-\theta_y)(\theta_x-\theta_z)\theta_z}{\theta_x},
\]
which do not come from the extended Weyl algebra.
\end{example}

\section{Rings of differential operators of quotient rings}
\label{sec:SmStmod} 
In this section, we consider a special class of toric face rings, namely quotients of normal affine semigroup rings by radical monomial ideals.
Our main goal is to characterize which differential operators on the quotient arise from operators on the ambient affine semigroup ring. 
We build off the techniques of \cite[Proposition 1.6]{SmStDO} and provide a careful inductive argument on the number of facets of the Newton polytope in order to generalize beyond the case where the ambient ring $S$ is regular.  
In Section~\ref{sec:Gorenstein}, we show that, if $S$ is regular (or even Gorenstein) and $J$ is the interior ideal of $S$, then $J\cD(S) = \cD(S, J)$, which gives a direct argument that our result agrees with \cite[Proposition 1.6]{SmStDO} in that case.

For any $\delta\in \cD(R)$, 
we will use $\delta J$ to denote the set consisting of products $\delta\circ f$ of differential operators for any $f \in J$ and $\delta \ast f$ to denote the action $\delta(f)$ to avoid any possible confusion.

\begin{prop}
\label{prop:SmStmod}
Let $R$ be commutative $\kk$-algebra and $J$ be an ideal of $R$. Let 
\[
\II(J) := \{\delta\in \cD(R)\mid \delta\ast J\subset J\}.
\]
The differential operators on $R$ that induce maps on $R/J$ are precisely those in $\II(J)$. \\
Further, there is an embedding of rings 
\[
\frac{\II(J)}{\cD(R,J)} 
\hookrightarrow
\cD(R/J).
\]
\end{prop}
\begin{proof}  
The first statement follows from the universal property of quotients. 
In fact, for any $\delta \in \cD(R)$, if $\delta$ induces an operator in $D(R/J)$, 
i.e., a map from $R/J$ to $R/J$, then we must have $\delta \ast J \subset J$. 
So, $\delta$ belongs to $\II(J)$ by definition.

Every operator in $\II(J)$ induces a differential operator from $R/J$ to itself. 
Now consider the map 
\[
\rho\colon \II(J) \rightarrow \cD(R/J)
\quad\text{given by}\quad \rho(\delta)=\delta'.
\]
The kernel is $\ker (\rho)=\{\delta \mid \delta \ast R \subseteq J\}=\cD(R,J)$ by \cite[1.2]{Mussontori}.  
Hence $\rho$ induces an injective map $\overline{\rho}\colon \II(J)/\cD(R,J) \hookrightarrow \cD(R/J)$ given by $\overline{\rho}(\overline{\delta})=\delta'$, as desired.  
\end{proof}

Before stating the main theorem, we will need some notation.  
Let $R$ be the normal affine semigroup ring defined by a matrix $A$. 
When referring to an arbitrary face or facet of $A$, we will use $\tau$ or $\sigma$, respectively.  
Every face of $A$ can be expressed as an intersection of facets. 
Recall the correspondence between the $\ZZ^d$-graded primes of $R$ and the faces of $A$. A face $\tau$ of $A$ corresponds to the prime ideal
$P_\tau:= 
\<t^{\bbm}\mid \bbm\in\NN A\setminus\NN\tau\>$.

Consider the $\ZZ^d$-graded radical monomial ideal $J:= \displaystyle \bigcap_{i=1}^r P_{\tau_i}$, with the face $\tau_i := \displaystyle \bigcap_{j=1}^{k_i}\sigma_{i,j}$ for facets $\sigma_{i,j}$ of $A$. 
For a fixed facet $\sigma$ of $A$, set 
\begin{equation}
\label{eqn:H}
H_{\sigma,{\bbm}}(\theta) := 
F_{\sigma}(\theta)+F_{\sigma}({\bbm}).
\end{equation}
When referring to a facet $\sigma_{i,j}$ from the definition of $J$, we will replace $\sigma$ in $F_\sigma$ and $H_{\sigma,{\bbm}}(\theta)$ with $i,j$, writing instead $F_{i,j}$ and $H_{i,j,{\bbm}}$. Lastly, for an integer $k$, let $[k]$ denote the set $\{1,2,\dots,k\}$, and let ${\bbj} = (j_1,j_2,\dots,j_r)\in K_J:= [k_1]\times[k_2]\times\cdots\times[k_r]$ denote an $r$-tuple of integers in the allowable range with respect to the ideal $J$. 

Finally whenever we have a product $L$ of linear factors, 
let $\rad(L)$ denote the (monic) generator of the radical of the ideal generated by $L$; in other words, $\rad(L)$ is a product of distinct linear polynomials in $\theta$.

\begin{thm}
\label{thm:normal-diffops}
Let $R$ be a normal affine semigroup ring defined by the $d \times n$ matrix $A$, 
and let $J = \displaystyle \bigcap_{i =1}^r P_{\tau_i}$ be the radical monomial ideal corresponding to the faces $\tau_i = \displaystyle\bigcap_{j=1}^{k_i} \sigma_{ij}$. 
Let $G_{\bbm}(\theta)$ and $H_{i,j,\bbm}(\theta)$ be as in \eqref{eqn:G} and \eqref{eqn:H}, respectively. 
Then for ${\bbm} \in\ZZ^d$, 
\begin{align*}
\left[\frac{\II(J)}{D(R,J)}\right]_{\bbm}
& \ = \ 
t^{\bbm}\cdot 
\dfrac{\left\< 
G_{\bbm}(\theta)\cdot 
\rad\left(
\displaystyle\prod_{i=1}^r
\displaystyle\prod_{F_{i,j}(\bbm)< 0} H_{{i,j},\bbm}(\theta)\right)
\ \bigg\vert \ 
{\bbj}\in K_J
\right\>}
{\left\< 
G_{\bbm}(\theta)\cdot 
\rad\left( 
\displaystyle\prod_{i=1}^r
\displaystyle\prod_{F_{i,j}(\bbm)\leq 0} 
H_{{i,j},\bbm}(\theta)\right)
\ \bigg\vert \ 
{\bbj}\in K_J
\right\>}, 
\end{align*}
which is precisely the contribution from $D(R)$ within the $\bbm$-th graded piece of $D(R/J)$ induced by the operators in $D(R)$. 
\end{thm}

To clarify notation, we will consider an example before proceeding with the proof of Theorem~\ref{thm:normal-diffops}.
\begin{example}
Let $R=\kk[x_1, x_2, x_3]$ and $J = \< x_1x_2,x_1x_3 \>$, in which case $r = 2$, $k_1 = 1$, $k_2 = 2$, and $K_J = \{1\} \times \{1,2\}$.  The prime decomposition of $J$ is $J =  \< x_1 \> \cap \< x_2,x_3 \> = P_{\sigma_{1,1}} \cap P_{\sigma_{2,1} \cap \sigma_{2,2}}$, and an element of $K_J$ corresponds to a choice of one prime from the set $\{\< x_1 \> \}$ and one prime from the set $\{\< x_2 \>, \< x_3 \> \}$. In the notation of Theorem~\ref{thm:normal-diffops},
\[
\II(J)_{\bbm}
= x^{\bbm} \cdot\left\<
G_{\bbm}(\theta) \prod_{\substack {F_{1,1}(\bbm)<0,\\ F_{2,1}(\bbm)<0}}  H_{1,1,\bbm}(\theta)
H_{2,1,\bbm}(\theta),\ \  G_{\bbm}(\theta) \prod_{\substack {F_{1,1}(\bbm)<0,\\ F_{2,2}(\bbm)<0}} H_{1,1,\bbm}(\theta)  H_{2,2,\bbm}(\theta)
\right\>.
\] 
It is the presence of $H_{1,1,\bbm}(\theta)$ as a factor of each generator $f$ of $\II(J)_{\bbm}$ that guarantees that whenever $x^{\bbm'} \in J$, we have $f \ast x^{\bbm'} \in \< x_1 \>$. Similarly, it is the presence of either $H_{2,1,\bbm}(\theta)$ or $H_{2,2,\bbm}(\theta)$ that guarantees that $f \ast x^{\bbm'} \in \< x_2,x_3 \>$.  Similarly, we have 
\[
D(R,J)_{\bbm}
= x^{\bbm} \cdot\left\<
G_{\bbm}(\theta) \prod_{\substack {F_{1,1}(\bbm)\leq 0,\\ F_{2,1}(\bbm)\leq 0}}  H_{1,1,\bbm}(\theta)
H_{2,1,\bbm}(\theta),\ \  G_{\bbm}(\theta) \prod_{\substack {F_{1,1}(\bbm)\leq 0,\\ F_{2,2}(\bbm)\leq 0}} H_{1,1,\bbm}(\theta)  H_{2,2,\bbm}(\theta)
\right\>.
\] 
As in the general formula, the difference between the computations of $\II(J)_{\bbm}$ and $D(R,J)_{\bbm}$ is seen in the difference between the strict inequalities $F_{i,j}(\bbm)< 0$ of $\II(J)_{\bbm}$ and the weak inequalities $F_{i,j}(\bbm) \leq 0$ of $D(R,J)_{\bbm}$.  Finally, the differential operators in $D(R)_{\bbm}$ which induce maps on $R/J$  are precisely those in $\II(J)_{\bbm}$ making $ \frac{\II(J)_{\bbm}}{D(R,J)_{\bbm}}$ a submodule of $D(R/J)_{\bbm}$, which respects the grading on numerator and denominator.
\end{example}

\begin{proof}[Proof of Theorem~\ref{thm:normal-diffops}]
Observe that 
\[
t^{\bbm}\cdot\II(\Omega({\bbm})) 
= t^{\bbm} \cdot\left\< \prod_{ F_\sigma(\bbm)<0} \prod_{i=0}^{-F_\sigma(\bbm)-1}(F_\sigma(\theta)-i) \right\>.
\]
Now for each $\bbm' \in \ZZ^d$, $t^{\bbm'} \in J$ if and only if $F_\sigma(\bbm')>0$ for all facets $\sigma$ of $A$.  
Because $\theta_i \ast t^{\bbm'} = m_i' t^{\bbm'}$ for each $i$, 
\[
\left[ t^{\bbm} \cdot
\prod_{ F_\sigma(\bbm)<0} 
\prod_{i=0}^{-F_\sigma(\bbm)-1}
(F_\sigma(\theta)-i) \right] 
\ast t^{\bbm'} 
= t^{\bbm+\bbm'} \cdot
\prod_{ F_\sigma(\bbm)<0} \prod_{i=0}^{-F_\sigma(\bbm)-1}
(F_\sigma(\bbm')-i).
\]
Hence, 
\[
\left[ t^{\bbm} \cdot G_{\bbm}(\theta) \right]
\ast t^{\bbm'} 
=t^{\bbm+\bbm'}\cdot G_{\bbm}(\bbm')
\in J
\]
exactly when at least one of the following two conditions is satisfied: \begin{enumerate}
\item $t^{\bbm+\bbm'} \in J$,\label{cond:powerinJ} or 
\item $ G_{\bbm}(\bbm')= 0$.\label{cond:Gm=0}
\end{enumerate}
First, for each $\bbm'$ such that $t^{\bbm'} \in J$, we must have that
$F_\sigma({\bbm'})+F_\sigma({\bbm}) 
= F_\sigma({\bbm+\bbm'})\geq 0$ 
for all facets $\sigma$ and that, 
for each $i \in [r]$, there exists some $j \in [k_i]$ so that $F_{i,j}({\bbm+\bbm'})>0$. 
We consider two conditions:
\begin{enumerate}[(i)]
\item ${\bbm}$ satisfies $F_{\sigma'}(\bbm)<0$ for some facet $\sigma'$, \label{cond:firstFacet} 

\item $F_{\sigma'}(\bbm)\geq 0$ for all facets $\sigma'$ and there exists some $i \in [r]$ for which $F_{i,j}(\bbm+\bbm')= 0$ for all $j \in [k_i]$. \label{cond:secondFacet}
\end{enumerate}

We claim that if either Condition \eqref{cond:firstFacet} or Condition \eqref{cond:secondFacet} holds, then there exists some $\bbm'$  with $t^{\bbm'} \in J$ and $G_{\bbm}(\bbm') \neq 0$. 
If Condition \eqref{cond:firstFacet} holds, then $G_{\bbm} (\bbm') \neq 0$ for exactly the $\bbm'$ that satisfy $F_{\sigma'}({\bbm})=-F_{\sigma'}({\bbm'})$ and $F_{\sigma}({\bbm'})\gg F_{\sigma}({\bbm})$ for all facets $\sigma\neq \sigma'$.
However, for all such ${\bbm}$, $H_{\sigma',{\bbm}}({\bbm'})G_{\bbm}({\bbm'})=0$ for all ${\bbm'}$ that satisfy $F_{\sigma'}({\bbm})=-F_{\sigma'}({\bbm'})$ and $F_{\sigma}({\bbm'})\gg -F_{\sigma}({\bbm})$ for all facets $\sigma\neq \sigma'$.

If Condition \eqref{cond:secondFacet} holds, then $G_{\bbm}(\bbm')$ will fail to vanish exactly for the ${\bbm'}$ for which there exists an $i$ with  $F_{i,j}({\bbm})=-F_{i,j}({\bbm'})$ for all $j\in[k_i]$
and $F_{i',j}({\bbm'})\gg F_{i',j}({\bbm})$ for all facets $\sigma_{i',j}$ with $i\neq i'$.  
However, for all such ${\bbm}$, 
$H_{i,j,{\bbm}}({\bbm'})G_{{\bbm}}({\bbm'})=0$ 
for all 
$i \in [r]$, $j \in [k_i]$, and ${\bbm'}$ 
satisfying the hypotheses that 
$F_{i,j}({\bbm})=-F_{i,j}({\bbm'})$ for all $j\in[k_i]$ 
and $F_{i',j}({\bbm'})\gg F_{i',j}({\bbm})$ for all facets $\sigma_{i',j}$ for $i\neq i'$.

Combining these calculations, 
\[
\II(J)_{\bbm} = t^{\bbm} \cdot\left\< 
G_{\bbm}(\theta)\cdot 
\rad\left(\prod_{i=1}^r
\prod_{F_{i,j}(\bbm)< 0} H_{i,j,\bbm}(\theta)
\right)
\ \bigg\vert \ 
{\bbj}\in K_J
\right\>.
\]

We compute $D(R,J)$ similarly. Now we begin with an arbitrary ${\bbm}\in\ZZ^d$ and $t^{\bbm'}\in R$. 
The only distinction between this calculation and the previous calculations for $\II(J)$ is that ${\bbm'}$ is now taken from a larger set. 
Namely, $F_\sigma({\bbm'})$ can now be $0$ as well. 
Then, Condition \eqref{cond:powerinJ} holds when $F_\sigma({\bbm}+{\bbm'})>0$ whenever $t^{\bbm'} \in R$, a condition automatically satisfied when $F_{\sigma}({\bbm})>0$. 
Again, in order to address Condition \eqref{cond:Gm=0}, we consider the two Conditions \eqref{cond:firstFacet} and \eqref{cond:secondFacet}, above. 
We note that, by the same argument used to compute $\II(J)$, if either Condition \eqref{cond:firstFacet} or Condition \eqref{cond:secondFacet} holds, then there exists some vector $\bbj \in K_J$ for which 
\[
G_{{\bbm}}({\bbm'})\cdot 
\rad\left(
\prod_{i=1}^r
\prod_{F_{i,j({\bbm})}< 0} H_{i,j,{\bbm}}({\bbm'})
\right)
\] 
is nonzero for some 
${\bbm'}$  with $t^{\bbm'} \in R$. 
Hence,
\[
D(R,J)_{\bbm} = t^{\bbm} \cdot\left\< 
G_{\bbm}(\theta)\cdot 
\rad\left(\prod_{i=1}^r
\prod_{F_{i,j}(\bbm)\leq 0} H_{i,j,\bbm}(\theta)
\right)
\ \bigg\vert \ 
{\bbj}\in K_J
\right\>.
\]

It now follows from Proposition~\ref{prop:SmStmod} that 
\begin{align*} 
\frac{\II(J)_{\bbm}}{D(R,J)_{\bbm}}
& \ = \ 
t^{\bbm}\cdot 
\dfrac{\left\< 
G_{m}(\theta)\cdot 
\rad\left(
\displaystyle\prod_{i=1}^r
\displaystyle\prod_{F_{i,j}(\bbm)< 0}
H_{{i,j},\bbm}(\theta)
\right)
\ \bigg\vert \ 
{\bbj}\in K_J
\right\>}
{\left\< 
G_{m}(\theta)\cdot 
\rad\left(
\displaystyle\prod_{i=1}^r
\displaystyle\prod_{F_{i,j}(\bbm)\leq 0} 
H_{{i,j},\bbm}(\theta)
\right)
\ \bigg\vert \ 
{\bbj}\in K_J
\right\>}, 
\end{align*}
as desired. 
\end{proof}

\section{Characterizing Gorenstein rings via differential operators}
\label{sec:Gorenstein} 
In this section, we compare $\cD(R, J)$ and $J\cD(R)$ when $R$ is a normal affine semigroup ring.  We find that the equality $\cD(R, J)=J\cD(R)$ holds if $J$ is a principal monomial ideal (see Proposition~\ref{prop:principalideals}). We then restrict to the special case of $J = \omega_R$, the intersection of all graded height one prime ideals of $R$.  That is, $\omega_R$ is the defining ideal of the union of all facets.  We choose the notation $\omega_R$ for this ideal, sometimes called the interior ideal, because it is a canonical module for $R$ (see, for example, \cite[Proposition~8.2.9]{CLSToricVarieties}).  Recall that $R$ is Gorenstein if and only if $\omega_R$ is principal (see, for example, \cite[Theorem~3.3.7]{CMRings}).  The main result of the section is as follows:

\begin{thm}
\label{thm:gor_char}
Let $R$ be a normal affine semigroup ring, and let $\omega_R$ be the intersection of all graded height one prime ideals of $R$. 
Then $R$ is Gorenstein if and only if $\omega_R D(R) = D(R, \omega_R) $.
\end{thm}

Before giving the proof of Theorem \ref{thm:gor_char}, we include two examples.

\begin{example}
\label{ex:rnc2}
Let $A = \left[\begin{smallmatrix}1&1&1\\ 0&1&2  \end{smallmatrix}\right]$,  $R = \kk[\NN A] = \kk [s,st,st^2]$, and $J=\omega_R = \<st\>$.  
Note that $R$ is Gorenstein since $\omega_R$ is principal.  
We will see in this case that $\omega_R \cD(R) = \cD(R,\omega_R)$. 

We denote the faces of $A$ by $\sigma_{1,1}$ and $\sigma_{2,1}$, and so that the primitive integral support functions are 
\[
F_{1,1}(\theta) = \theta_2
\quad\text{and}\quad 
F_{2,1}(\theta) = 2\theta_1 - \theta_2. 
\]
Then recall that
\[
H_{1,1,\bbm}(\theta) := F_{1,1}(\theta) + F_{1,1}(\bbm) \quad\text{and}\quad H_{2,1,\bbm}(\theta) := F_{2,1}(\theta) + F_{2,1}(\bbm)
\]
and
\[
G_{\bbm}(\theta) := \prod_{i=0}^{-F_{1,1}(\bbm) - 1} (F_{1,1}(\theta) -i) \prod_{j=0}^{-F_{2,1}(\bbm) -1} (F_{2,1}(\theta) - j).
\]

By (the proof of) Theorem~\ref{thm:normal-diffops}, 
\begin{align*}
\cD(R,\omega_{R})_{\bbm} 
&= s^{m_1}t^{m_2} \left\< 
G_{\bbm}(\theta)\cdot 
\left( 
\displaystyle\prod_{i=1}^2
\displaystyle\prod_{F_{i,1}(\bbm)\leq 0} 
H_{{i,1},\bbm}(\theta)\right)\right\>   \\
&= s^{m_1}t^{m_2}  
\left\< \prod_{i=0}^{-F_{1,1}(\bbm)} (F_{1,1} (\theta) - i)   \prod_{j=0}^{-F_{2,1}(\bbm)} (F_{2,1}(\theta) - j) \right\>\\
&{= s^{m_1}t^{m_2}  
\left\< \prod_{i=0}^{-m_2} (\theta_2 - i)   \prod_{j=0}^{-2m_1+m_2} (2\theta_1-\theta_2 - j) \right\>},
\end{align*}
where, by convention, an empty product is $1$ and $\left<- \right>$ denotes an ideal in $\kk [\theta_1,\theta_2]$.  
Multiplying the expression for $\cD(R)$ given in Theorem~\ref{thm:DRd} by $\omega_{R}$, we obtain 
\[ 
\omega_{R} \cD(R) = 
\<st\>\cdot 
\bigoplus_{m \in \ZZ^2} s^{m_1}t^{m_2} \cdot 
\langle G_{\bbm}(\theta)\rangle.
\] 
Denote ${\bf 1} := (1,1)$.  Then we have
\begin{align*}
(\omega_{R} \cD(R))_{\bbm} 
&= s^{m_1}t^{m_2} \cdot 
\langle G_{\bbm -{\bf 1}}(\theta) \rangle \\
&= s^{m_1}t^{m_2}\cdot \left< \prod_{i=0}^{-F_{1,1}(\bbm - {\bf 1}) -1} (F_{1,1} (\theta) - i)\prod_{i=0}^{-F_{2,1}(\bbm - {\bf 1})-1} (F_{2,1}(\theta) - i) \right> \\
&= s^{m_1}t^{m_2}\cdot  \left<\prod_{i=0}^{-F_{1,1}({\bbm}) + F_{1,1}({\bf 1}) -1} (F_{1,1} (\theta) - i)\prod_{i=0}^{-F_{2,1}({\bbm}) + F_{2,1}({\bf 1}) -1}(F_{2,1}(\theta) - i) \right> \\
&= s^{m_1}t^{m_2} \cdot\left< \prod_{i=0}^{-F_{1,1}({\bbm})} (F_{1,1} (\theta) - i) \prod_{i=0}^{-F_{2,1}({\bbm})} (F_{2,1}(\theta) - i) \right>\\
&{= s^{m_1}t^{m_2}  
\left\< \prod_{i=0}^{-m_2} (\theta_2 - i)   \prod_{j=0}^{-2m_1+m_2} (2\theta_1-\theta_2 - j) \right\>.}
\end{align*}
It follows that $\omega_R \cD(R) = \cD(R,\omega_R)$.
\end{example}

\begin{example}
\label{ex:rnc3}
Let $A = \left[\begin{smallmatrix}1&1&1&1\\ 0&1&2&3  \end{smallmatrix}\right]$,  $R = \kk[\NN A] = \kk [s,st,st^2,st^3]$, and $J=\omega_R = \<st,st^2\>$.  
Note that $R$ is not Gorenstein. 
We denote the faces of $A$ by $\sigma_{1,1}$ and $\sigma_{2,1}$, and so that the primitive integral support functions are 
\[
F_{1,1}(\theta) = \theta_2
\quad\text{and}\quad 
F_{2,1}(\theta) = 3\theta_1 - \theta_2. 
\]
Then for $\bbm = (m_1,m_2)$, we have
\[
H_{1,1,\bbm}(\theta) = F_{1,1}(\theta) + F_{1,1}(\bbm) \quad\text{and}\quad H_{2,1,\bbm}(\theta) = F_{2,1}(\theta) + F_{2,1}(\bbm)
\]
and
\[
G_{\bbm}(\theta) = \prod_{i=0}^{-F_{1,1}(\bbm) - 1} (F_{1,1}(\theta) -i) \prod_{i=0}^{-F_{2,1}(\bbm) -1} (F_{2,1}(\theta) - i).
\]

Multiplying the expression for $\cD(R)$ given in Theorem~\ref{thm:DRd} by $\omega_R$, we obtain
\[ 
\omega_R \cD(R) = 
\<st,st^2\>\cdot 
\bigoplus_{m \in \ZZ^2} s^{m_1}t^{m_2} \cdot\langle G_{\bbm}(\theta) \rangle, 
\]
so that
\[ 
(\omega_R \cD(R))_{\bbm} = 
 s^{m_1}t^{m_2} \cdot
\left\langle G_{\bbm - (1,1)}(\theta), G_{\bbm - (1,2)}(\theta)\right\rangle, 
\]
where 
\begin{align*}G_{\bbm - (1,1)}(\theta)& = \prod\limits_{j=0}^{-F_{1,1}({\bbm}-(1,1))-1 } (F_{1,1}(\theta)-j)\prod\limits_{j=0}^{-F_{2,1}({\bbm}-(1,1))-1} (F_{2,1}(\theta)-j)\\
&=\prod\limits_{j=0}^{-F_{1,1}({\bbm}) } (F_{1,1}(\theta)-j)\prod\limits_{j=0}^{-F_{2,1}({\bbm})+1} (F_{2,1}(\theta)-j)\\
&{=\prod\limits_{j=0}^{-m_2 } (\theta_2-j)\prod\limits_{j=0}^{-3m_1+m_2+1} (3\theta_2-\theta_1-j)}
\end{align*}

when $F_{1,1}({\bbm}) {=m_2}\leq 0 \text{ and }  F_{2,1}({\bbm}){=3m_1-m_2 }\leq 1$ and 
\begin{align*}G_{\bbm - (1,2)}(\theta) 
&=\prod\limits_{j=0}^{-F_{1,1}({\bbm}-(1,2))-1} (F_{1,1}(\theta)-j)\prod\limits_{j=0}^{-F_{2,1}({\bbm}-(1,2))-1 } (F_{2,1}(\theta)-j)\\
&=\prod\limits_{j=0}^{-F_{1,1}({\bbm})+1} (F_{1,1}(\theta)-j)\prod\limits_{j=0}^{-F_{2,1}({\bbm}) } (F_{2,1}(\theta)-j)\\
&{=\prod\limits_{j=0}^{-m_2+1} (\theta_2-j)\prod\limits_{j=0}^{-3m_1+m_2 } (3\theta_1-\theta_2-j)}\\
\end{align*}
when $ F_{1,1}({\bbm}){=m_2} \leq 1 \text{ and }  F_{2,1}({\bbm}){=3m_1-m_2} \leq 0$.

On the other hand, by Theorem~\ref{thm:normal-diffops} and  a computation similar to the one in Example~\ref{ex:rnc2},
\begin{align*}
\cD(R, \omega_R)_{\bbm} &= 
s^{m_1}t^{m_2} \cdot 
\left\< \prod_{i=0}^{-F_{1,1}({\bbm})} (F_{1,1} (\theta) - i) \prod_{i=0}^{-F_{2,1}({\bbm})}
 (F_{2,1}(\theta) - i) \right\>\\
&{= 
s^{m_1}t^{m_2} \cdot 
\left\< \prod_{i=0}^{-m_2} (\theta_2- i) \prod_{i=0}^{-3m_1+m_2}
 (3\theta_1-\theta_2 - i) \right\>}\\
\end{align*}
when $ F_{1,1}({\bbm}){=m_2} \leq 0 \text{ and }  F_{2,1}({\bbm}) {=3m_1-m_2}\leq 0$.

To see $\omega_R \cD(R)$ and $\cD(R, \omega_R)$ are different, consider any degree, say $\bbm = (-1,-1)$.  Note that the degree $\bbm$ piece of $\cD(R, \omega_R)$ is generated by one element, namely
\[
s^{-1}t^{-1} \prod_{i=0}^1 (F_{1,1}(\theta) - i) \prod_{i=0}^2 (F_{2,1}(\theta)-i)
\]
On the other hand, the degree $\bbm$ piece of $\omega_R\cD(R)$ is generated by two elements; it is generated in $\kk [\theta_1,\theta_2]$ by
\[
s^{-1}t^{-1} \prod_{i=0}^1 (F_{1,1}(\theta) - i) \prod_{i=0}^3 (F_{2,1}(\theta)-i) \quad\text{and}\quad s^{-1}t^{-1} \prod_{i=0}^2 (F_{1,1}(\theta) - i) \prod_{i=0}^2 (F_{2,1}(\theta)-i)
\]
In particular,
\[
s^{-1}t^{-1} \prod_{i=0}^1 (F_{1,1}(\theta) - i) \prod_{i=0}^2 (F_{2,1}(\theta)-i) \in D(R,\omega_R)
\]
but
\[
s^{-1}t^{-1} \prod_{i=0}^1 (F_{1,1}(\theta) - i) \prod_{i=0}^2 (F_{2,1}(\theta)-i) \notin \omega_R\cD(R).
\qedhere
\]
\end{example}

In order to prove Theorem~\ref{thm:gor_char}, we will need some preliminary results. 

\begin{lem}\cite[Theorem 6.33]{Bruns-Gub}
\label{lem:+Claim}
Let $R$ be a normal affine semigroup ring of dimension $d$, and let $F_{\sigma_1}, \dots, F_{\sigma_{r}}$ be the support functions of the facets ${\sigma_i}$ of the semigroup defining $R$. 
Let $\omega_R$ be the intersection of all graded height one prime ideals of $R$. 
Then $\omega_R$ is a principal ideal if and only if 
there exists $t^{\bbc} \in \omega$ for some ${\bbc} \in \ZZ^d$ such that
$F_{\sigma_i}({\bbc}) =1$
for all $i = 1, \dots, k$.
\end{lem}

We first show that principal monomial ideals $J$ behave well (i.e., $J\cD(R)=\cD(R,J)$) for affine normal semigroup rings.

\begin{prop} 
\label{prop:principalideals}
Let $R$ be a normal affine semigroup ring and let $J=\<t^{\bbc}\>$, 
then 
\[
J \cD(R) = \cD(R, J).
\]
\end{prop}

\begin{proof}
Since $J \cD(R) \subseteq \cD(R, J)$, it suffices to prove that $ J\cD(R)_{\bbm} = \cD( R, J)_{\bbm}$ for all multidegrees $\bbm$.
We do so by showing that they are both cyclic and are generated by the same differential operator.  

Since $J= \< t^{\bbc} \>$,
by the expression obtained in Theorem~\ref{thm:DRd}, we have 
\[ 
J \cD(R)_{\bbm} = 
t^{\bbm} \cD(R) _{\bbm-\bbc } = t^{\bbm} \left\< 
G_{{\bbm-\bbc}}(\theta)\right\>. 
\]
Since $F_{i}$ is linear, we have 
\[ 
F_{i}( - ( {\bbm -\bbc} )) -F_{i}(\bbc) = - F_{i}({\bbm} ) + F_{i}({\bbc} ) - F_{i}(\bbc) 
= - F_{i}({\bbm}). 
\] 
Hence the single generator for the graded piece $J \cD(R)_{\bbm}$ can be viewed as 
\[
t^{\bbm} \cdot 
\left\< G_{\bbm}(\theta)\prod_{i = 1}^r H_{i,{\bbm}}(\theta)\right\>,
\]
which is exactly the generator of $\cD(R, J)_{\bbm}$, completing the proof.
\end{proof}

Now we are ready to prove the main theorem of the section: Theorem~\ref{thm:gor_char}. As above, we will write $F_i$ and $H_{i,{\bbm}}$ for $F_{\sigma_i}$ and $H_{\sigma_i,{\bbm}}$, respectively.

\begin{proof}[Proof of Theorem \ref{thm:gor_char}]
Suppose first that $R$ is Gorenstein.  Then $\omega_R$ is principal, say $\omega_R= \< t^{\bbc} \>$. Now by Proposition~\ref{prop:principalideals}, we obtain $\omega_R\cD(R)=\cD(R, \omega_R)$.

Conversely assume that $R$ is not Gorenstein, that is, $\omega_R$ is not principal. Since $\cD(R, \omega_R)$ and $ \omega_R \cD(R)$ are both multigraded and it is always true that $\omega_R \cD(R) \subseteq D(R, \omega_R)$, 
we may prove that they are not equal by identifying their difference at $\bbm =\boldzero$. 

As mentioned above, $\cD(R, \omega_R)_{\bbm} $ is generated by a single element
${t^{\bbm}} 
\left\langle G_{{\bbm}}(\theta)\prod_{i = 1}^r H_{i,{\bbm}}(\theta)\right\rangle$. 
Thus  
\[ 
\cD(R, \omega_R)_{\boldzero} 
= 
\left\< \prod_{ i = 1}^r 
H_{i,{\boldzero}}(\theta) \right\> 
= \left\< F_{1}(\theta) \cdots F_{r}(\theta) \right\>. 
\]  
Let $\omega_R$ be minimally generated by 
${t}^{{\bbc}_1}, {t}^{{\bbc}_2}, \dots, {t}^{{\bbc}_h}$. Then  
\[ 
( \omega \cD(R) )_{\bbm}  
= \sum_{j=1}^h  {t}^{ {\bbc}_j }  {t}^{ {\bbm} -  {\bbc}_j } \cdot\cD(R) _{{\bbm}  -{\bbc}_j  } 
= \sum_{j=1}^h  {t}^{ \bbm}\cdot \left\< 
G_{{\bbm-\bbc_j}}(\theta) \right\> 
=  {t}^{ \bbm}\cdot \left\< 
G_{{\bbm-\bbc_j}}(\theta) \ \bigg\vert \  j=1,2,\dots,h \right\>.
\]
In particular, when  ${\bbm} = \boldzero$, 
\[
( \omega \cD(R) )_{\boldzero} =   \left\< 
G_{-{\bbc_j}}(\theta)
\ \bigg\vert \  j=1,2,\dots,h \right\>.
\]
We know from Lemma~\ref{lem:+Claim} that some $F_{\sigma_i}({\bbc}_j) \neq 1$. 
Therefore 
\begin{align} 
\label{eq:omega-zero-ideal}
( \omega \cD(R) )_{\boldzero} 
& =    \sum_{j=1}^h  \left\langle \prod_{ i = 1}^r  F_{i}(\theta)(F_{i}(\theta)-1) \cdots (F_{i}(\theta)-(F_{i}(c_j)-1)) \right\rangle \\
& = 
\left\< F_{1}(\theta)F_{2}(\theta) \cdots F_{r}(\theta) \right\> 
\cap 
\sum_{j=1}^h  \left\< \prod_{i=1}^r  (F_{i}(\theta)-1) \cdots (F_{i}(\theta) - ( F_{i}( {\bbc}_j )-1 ) )  \right\>
\end{align}

We claim that the ideal
\begin{align}
\label{eq:proper-ideal}
\sum_{j=1}^h  \left\< \prod_{i=1}^r  (F_{i}(\theta)-1) \cdots (F_{i}(\theta) - ( F_{i}( {\bbc}_j )-1 ) )  \right\>
\end{align}
from \eqref{eq:omega-zero-ideal} is a proper ideal of $\kk[\theta]$. 
Once this is established, it will follow from \eqref{eq:omega-zero-ideal} that the inclusion 
\[
( \omega_R \cD(R) )_{\boldzero}  \subseteq D(R, \omega_R)_{\boldzero} = \< F_1(\theta)\cdots F_r(\theta)\>,
\]
is proper, which will conclude the proof. 

To see that \eqref{eq:proper-ideal} is a proper ideal of $\kk[\theta]$, we may assume after re-ordering that $\bbc_1$ generates a ray of the cone 
\[
\cC 
:= \RR_{\geq 0}\{\bbc_1,\bbc_2,\dots,\bbc_h\}
= \{ \bbm\in\NN A\mid F_1(\bbm)\geq 1,\dots,F_r(\bbm)\geq 1\} 
\subseteq \RR^d
\]
over the exponents of the minimal generators of $\omega_R$
and, in light of Lemma~\ref{lem:+Claim}, that there are $d-1$ linearly independent primitive integral support functions, which after re-labelling can be taken as $F_1,\dots, F_{d-1}$, such that 
\[
F_1(\bbc_1)=F_2(\bbc_1)=\cdots= F_{d-1}(\bbc_1)=1 
\quad\text{and}\quad 
F_d(\bbc_1)>1. 
\]
(There may be more $i$ such that $F_i(\bbc_1)=1$, but we need to make use of only a linearly independent set of them, which is necessarily of size $d-1$, since $\bbc_1$ generates a ray of $\cC$.) 

\begin{example}
We pause our proof to give an example illustrating the newly-introduced notation.
Consider the case $d=3$ and refer to Figure~\ref{fig:canmodnotGor} as a visualization. 
Fix $\bbc_1$ to be a vertex of the convex hull of the $\bbc_i$, or, rather, a generator of a ray of the Newton polyhedron of $\omega_R$. 
Without loss of generality, 
\[
F_1(\bbc_1)=1=F_2(\bbc_1)
\quad\text{and}\quad 
F_3(\bbc_1)>1.
\]
Further, there exists an $r$ such that 
\[
F_1(\bbc_i)>1 
\quad\text{for all}\quad 
r+1\leq i\leq h
\quad\text{and}\quad 
F_1(\bbc_i)=1 
\quad\text{for all}\quad 
1\leq i\leq r.
\]
Note that in Figure~\ref{fig:canmodnotGor}, we have chosen $r=3$.
Now, since $\bbc_1$ is a vertex, it must be that 
\[
F_2(\bbc_i)>1 
\quad\text{for all}\quad 
2\leq i\leq r.
\]
Hence, $\<F_1(\theta)-1,F_2(\theta)-1,F_3(\theta)-1\>$ is a primary component of \eqref{eq:proper-ideal}.
\vspace*{-4mm}
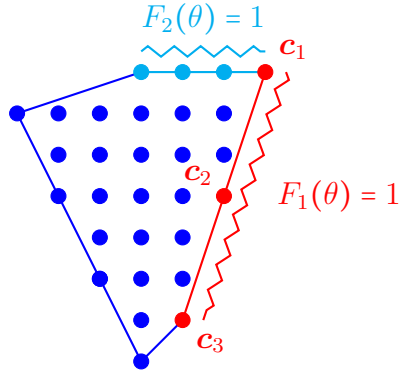
\begin{figure}[ht]
\centering
\begin{tikzpicture}[scale=0.55]
\draw[blue, thick] (0,0) -- (1,1);
\draw[blue, thick] (0,0) -- (-3,6);
\draw[blue, thick] (-3,6) -- (0,7);
\draw[cyan, thick] (0,7) -- (3,7);
\draw[decorate,decoration={zigzag,amplitude=.7mm},cyan, thick] (0,7.5) -- node[above=2pt]{$F_2(\theta)=1$} (3,7.5);
\draw[red, thick] (3,7) -- (1,1);
\draw[decorate,decoration={zigzag,amplitude=.7mm},red, thick] (3.5,7) -- node[right=8pt]{$F_1(\theta)=1$} (1.5,1);
\foreach \x in {0,1,2}{ \node[draw,circle,inner sep=2pt,cyan,fill] at (\x,7) {};}
\foreach \x in {1,2,3}{ \node[draw,circle,inner sep=2pt,red,fill] at (\x,3*\x-2) {};}
\foreach \x in {-2,-1,...,2}{
      \foreach \y in {5,6}{
        \node[draw,circle,inner sep=2pt,blue,fill] at (\x,\y) {};}}
\foreach \x in {-1,0,1}{
      \foreach \y in {2,3,4}{
        \node[draw,circle,inner sep=2pt,blue,fill] at (\x,\y) {};}}
\foreach \y in {0,1}{ \node[draw,circle,inner sep=2pt,blue,fill] at (0,\y) {};}
\foreach \x in {-3,-2,-1}{ \node[draw,circle,inner sep=2pt,blue,fill] at (\x,-2*\x) {};}
\filldraw[red] (3,7) circle (2pt) node[above right=1pt]{${\bbc_1}$};
\filldraw[red] (2,4) circle (2pt) node[above left=.2pt]{${\bbc_2}$};
\filldraw[red] (1,1) circle (2pt) node[below right=1pt]{${\bbc_3}$};
\end{tikzpicture}
\caption{Hypothetical generators of $\omega_R$ in dimension 3}
\label{fig:canmodnotGor}
\end{figure}
\end{example}

We resume our proof in full generality.  Given that $\bbc_1$ is a ray of the cone $\cC$ with $F_i(\bbc_1)=1$ for every $i$ with $1\leq i\leq d-1$, it follows that, for every $j>1$, there is an $i_j$ with $1\leq i_j\leq d-1$ such that $F_{i_j}(\bbc_j)>1$. 
Now set $i_1 = d$ because $F_d(\bbc_1)>1$. 
Then since $1\leq i_j\leq d$ for all $j$ and the hyperplanes defined by $\{F_i(\theta)=0\}_{i=1}^{d}$ are necessarily non-parallel, the ideal
\[
\< F_{i_1}(\theta) -1, F_{i_2}(\theta)-1,\ldots, F_{i_d}(\theta)-1 \>
\]
is a proper, primary (in fact, prime) component of \eqref{eq:proper-ideal}. 
Thus \eqref{eq:proper-ideal} is a proper ideal of $\kk[\theta]$, as desired to establish the final needed claim. 
\end{proof}

\bibliographystyle{amsalpha}
\bibliography{refs}
\end{document}